\documentclass{llncs}
\usepackage[noend]{algpseudocode}
\usepackage{algorithm}
\usepackage{times}
\usepackage{amssymb}
\usepackage{amsmath}
\usepackage{latexsym}
\usepackage{graphics}
\usepackage{url}
\usepackage{verbatim}
\usepackage{color}
\usepackage{setspace}
\usepackage{multirow}
\usepackage{lineno}
\usepackage{booktabs}
\usepackage{graphicx}
\usepackage{caption}
\usepackage[justification=centering]{subcaption}
\usepackage{enumitem}
\usepackage{tikz}
\usepackage{ctable}
\usepackage{marvosym}
\algnewcommand\And{\textbf{and}}
\algnewcommand\Or{\textbf{or}}
\algnewcommand\Not{\textbf{not}}
\algnewcommand\In{\textbf{in}}
\algnewcommand\Each{\textbf{each}}

\newcommand{\pii}{\pi}

\newcommand{\squishlist}{
 \begin{list}{$\bullet$}
  { \setlength{\itemsep}{0pt}
     \setlength{\parsep}{3pt}
     \setlength{\topsep}{3pt}
     \setlength{\partopsep}{0pt}
     \setlength{\leftmargin}{2.5em}
     \setlength{\labelwidth}{1em}
     \setlength{\labelsep}{0.5em} } }

\newcommand{\squishlisttwo}{
 \begin{list}{$\triangleright$}
  { \setlength{\itemsep}{0pt}
     \setlength{\parsep}{0pt}
    \setlength{\topsep}{0pt}
    \setlength{\partopsep}{0pt}
    \setlength{\leftmargin}{2em}
    \setlength{\labelwidth}{1.5em}
    \setlength{\labelsep}{0.5em} } }

\newcommand{\squishend}{
  \end{list}  }

\usepackage{listings}

\definecolor{verbgray}{gray}{0.9}

\lstnewenvironment{code}{%
  \lstset{backgroundcolor=\color{verbgray},
 frame=single,
  language=C,
  framerule=0pt,
  basicstyle=\ttfamily,
  showstringspaces=false,
  commentstyle=\color{blue}\textit,
  keywordstyle=\color{black}\bf,
  numbers=none,
  keepspaces=true,
  columns=fullflexible}}{}

\definecolor{shadecolor}{rgb}{.91, .91, .91}
\definecolor{bordercolor}{rgb}{.8, .8, .6}

\definecolor{ultramarine}{rgb}{0, 0.125, 0.376}

 \definecolor{arsenic}{rgb}{0.23, 0.27, 0.29}
 \definecolor{beige}{rgb}{0.96, 0.96, 0.86}

\definecolor{amber}{rgb}{1.0, 0.75, 0.0}
\definecolor{orange}{rgb}{1.0, 0.49, 0.0}
\definecolor{dandelion}{rgb}{0.94, 0.88, 0.19}

  \definecolor{indiagreen}{rgb}{0.07, 0.53, 0.03}
  \definecolor{huntergreen}{rgb}{0.21, 0.37, 0.23}

\newcommand{\blue}[1] {\textcolor{blue}{#1}}
\newcommand{\red}[1] {\textcolor{red}{#1}}
\newcommand{\green}[1] {\textcolor{green}{#1}}

\newcommand{\bblue}[1] {{\bf \textcolor{blue}{#1}}}
\newcommand{\rred}[1] {{\bf \textcolor{red}{#1}}}

\definecolor{shadecolor}{rgb}{.9, .9, .9}
\definecolor{exframecolor}{rgb}{.4, .4, .4}

 \usepackage{framed}

 \colorlet{framecolor}{ultramarine}
    {\endMakeFramed}

    \newenvironment{frshaded*}{%
    \MakeFramed {\advance\hsize-\width \FrameRestore}}%
    {\endMakeFramed}

    \newcounter{examplecounter}
\newenvironment{exam}{
 \begin{frshaded*}
    \refstepcounter{examplecounter}%
    \noindent
  \textbf{Example \arabic{examplecounter}}%
  \quad
}{%
\end{frshaded*}
}

{\endMakeFramed}

\newenvironment{frshaded2*}{%
    \MakeFramed {\advance\hsize-\width \FrameRestore}}%
    {\endMakeFramed}

\newtheorem{sRule}[subsection]{Rule}
{\begin{sRule}[#1] \hspace*{\fill} \\ Input: #2. \\ Output: }%
{\end{sRule}}

\newcommand{\PERMS}[1]{\mathbb{P}(#1)}
\newcommand{\BPERMS}[1]{\mathbb{\overline{P}}(#1)}
\newcommand{\CPERMS}[2]{\mathbb{P}_{#2}(#1)}

\newcommand{\PGraph}[1]{\mathbb{G}(#1)}
\newcommand{\BGraph}[1]{\mathbb{\overline{G}}(#1)}
\newcommand{\CGraph}[2]{\mathbb{G}_{#2}(#1)}

\newcommand{\cflip}[1]{\mathsf{flip}_{(#1)}}
\newcommand{\flip}[1]{\mathsf{flip}_{#1}}
\newcommand{\flipper}[2]{\flip{#2}(#1)}

\newcommand{\rev}[1]{\overleftarrow{#1}}

\newcommand{\GreedyCMin}[2]{\mathsf{GreedyMin}_{#2}(#1)}

\newcommand{\RecCMin}[2]{\mathbf{Rec}_{#2}(#1)}

\newcommand{\SEQQ}{\sigma}

\newcommand{\RANK}[1]{\mathsf{Rank}(#1)}

\newcommand{\suc}[1]{\mathsf{succ}(#1)}

\newcommand{\C}[2]{%
\ifnum#2=1{\bf #1}\fi
\ifnum#2=2{\bf \red{#1}}\fi
\ifnum#2=3{{\bf \blue{#1}}}\fi
\ifnum#2=4{\green{#1}}\fi
}

\usepackage{tikz}

\newcounter{stackCounter}

\newcommand{\reverseList}[1]{
  \let\reversedList\empty
  \foreach\x in {#1} {
    \ifx\reversedList\empty
      \xdef\reversedList{\x}%
    \else
      \xdef\reversedList{\x,\reversedList}%
    \fi
  }
}

\newcommand{\stackBurnt}[2][1]{%
  \colorlet{fullBurnt}{brown!55!black}%
  \colorlet{lessBurnt}{fullBurnt!75!white}%
  \tikzstyle{top}=[ultra thin, fill=brown!58]%
  \tikzstyle{bottom}=[ultra thin, fill=lessBurnt]%
  \pgfmathsetmacro\pancakeHeight{0.1}%
  \pgfmathsetmacro\pancakeWidthMin{0.2}%
  \pgfmathsetmacro\pancakeWidthStep{0.05}%
  \pgfmathsetmacro\pancakeThickness{0.025}%
  \pgfmathsetmacro\stackStep{0.1}%
  \begin{tikzpicture}%
    \begin{scope}[scale=#1]%
      \setcounter{stackCounter}{0}%
      \reverseList{#2}%
      \foreach \x in \reversedList {%
        \pgfmathsetmacro\yBottom{\thestackCounter * \stackStep}%
        \pgfmathsetmacro\yTop{\thestackCounter * \stackStep + \pancakeThickness}%
        \pgfmathabs{\x}%
        \pgfmathsetmacro\xWidth{\pancakeWidthMin + (\pgfmathresult-1) * \pancakeWidthStep}%
        \pgfmathparse{ifthenelse(\x>=0,"bottom","top")}%
        \draw[style=\pgfmathresult] (0,\yBottom) ellipse ({\xWidth} and {\pancakeHeight});
        \pgfmathparse{ifthenelse(\x>=0,"top","bottom")}%
        \draw[style=\pgfmathresult] (0,\yTop) ellipse ({\xWidth} and {\pancakeHeight});
        \stepcounter{stackCounter}%
      }%
    \end{scope}%
  \end{tikzpicture}%
}%

\newcommand{\stack}[2][1]{%
  \tikzstyle{top}=[ultra thin, fill=brown!50]%
  \tikzstyle{bottom}=[ultra thin, fill=brown!58]%
  \pgfmathsetmacro\pancakeHeight{0.1}%
  \pgfmathsetmacro\pancakeWidthMin{0.2}%
  \pgfmathsetmacro\pancakeWidthStep{0.05}%
  \pgfmathsetmacro\pancakeThickness{0.025}%
  \pgfmathsetmacro\stackStep{0.1}%
  \begin{tikzpicture}%
    \begin{scope}[scale=#1]%
      \setcounter{stackCounter}{0}%
      \reverseList{#2}%
      \foreach \x in \reversedList {%
        \pgfmathsetmacro\yBottom{\thestackCounter * \stackStep}%
        \pgfmathsetmacro\yTop{\thestackCounter * \stackStep + \pancakeThickness}%
        \pgfmathabs{\x}%
        \pgfmathsetmacro\xWidth{\pancakeWidthMin + (\pgfmathresult-1) * \pancakeWidthStep}%
        \draw[style=bottom] (0,\yBottom) ellipse ({\xWidth} and {\pancakeHeight});
        \draw[style=top] (0,\yTop) ellipse ({\xWidth} and {\pancakeHeight});
        \stepcounter{stackCounter}%
      }%
    \end{scope}%
  \end{tikzpicture}%
}%

\begin{document}

\title{A Hamilton Cycle in the $k$-Sided Pancake Network}
\author{
B. Cameron
%\thanks{School of Computer Science, University ofGuelph, Canada.
%\texttt{ email:ben.cameron@uoguelph.ca}
%}
\ \ \   \
J. Sawada%\thanks{School of Computer Science, University ofGuelph, Canada.
 %\texttt{ email:jsawada@uoguelph.ca}
% } 
\ \ \   \
  A. Williams
  %\thanks{Computer Science, Williams College, Williamstown, USA.
  %\texttt{E-mail:} \texttt{aaron.williams@williams.edu} }
}

\institute{}

\date{\today}
\maketitle

% Aaron crunching a bit of space.
\begin{abstract}
We present a Hamilton cycle in the $k$-sided pancake network and four combinatorial algorithms to traverse the cycle.
The network's vertices are coloured permutations $\pi = p_1p_2\cdots p_n$, where each $p_i$ has an associated colour in $\{0,1,\ldots, k{-}1\}$.
There is a directed edge $(\pi_1,\pi_2)$ if $\pi_2$ can be obtained from $\pi_1$ by a ``flip'' of length $j$, which reverses the first $j$ elements and increments their colour modulo $k$.
Our particular cycle is created using a greedy min-flip strategy, and the average flip length of the edges we use is bounded by a  constant. %$e + \epsilon$.
% Hamilton cycles in this network provide cyclic Gray codes of coloured permutations, where successive permutations differ by a single flip.
% Since short flips are efficient, we are able to generate the corresponding Gray code efficiently.
By reinterpreting the order recursively, we can generate successive coloured permutations in $O(1)$-amortized time, or each successive flip by a loop-free algorithm.
We also show how to compute the successor of any coloured permutation in $O(n)$-time. %, which allows the cycle to be traversed by starting at an arbitrary permutation.
Our greedy min-flip construction generalizes known Hamilton cycles for the pancake network (where $k=1$) and the burnt pancake network (where $k=2$).
Interestingly, a greedy max-flip strategy works on the pancake and burnt pancake networks, but it does not work on the $k$-sided network when $k>2$.
%Our results also represent the first efficient Gray code algorithm for coloured permutations.
\end{abstract}

\section{Introduction}

%\subsection{Pancake Flipping}

Many readers will be familiar with the story of Harry Dweighter, the harried waiter who sorts stacks of pancakes for his customers.
He does this by repeatedly grabbing some number of pancakes from the top of the stack, and flipping them over.
For example, if the chef in the kitchen creates the stack \stack[0.6]{4,3,1,2}, then Harry can sort it by flipping over all four pancakes \stack[0.6]{2,1,3,4}, and then the top two \stack[0.6]{1,2,3,4}.

This story came from the imagination of Jacob E. Goodman \cite{Goodman}, who was inspired by sorting folded towels \cite{Newspaper}.
%,Je used the pseudonym Harry Dweighter to enchance the tale.
His original interest was an upper bound on the number of flips required to sort a stack of $n$ pancakes.
Despite its whimsical origins, the problem attacted interest from many mathematicians and computer scientists, including a young Bill Gates \cite{Gates}.
Eventually, it also found serious applications, including genomics~\cite{Gene}.

A variation of the original story involves burnt pancakes.
In this case, each pancake has two distinct sides: burnt and unburnt.
When Harry flips the pancakes, the pancakes involved in the flip also turn over, and Harry wants to sort the pancakes so that the unburnt sides are facing up.
For example, Harry could sort the stack \stackBurnt[0.6]{-4,-3,-1,2} by flipping all four \stackBurnt[0.6]{-2,1,3,4}, then the top two \stackBurnt[0.6]{-1,2,3,4}, and the top one \stackBurnt[0.6]{1,2,3,4}.
Similar lines of research developed around this problem (e.g. \cite{Cohen}, \cite{Gene}).
The physical model breaks down beyond two sides, however, many of the same applications do generalize to ``$k$-sided pancakes''.

% Unfortunately, the physical model break down beyond two sides --- triangular pancakes anyone? --- however, many of the same applications do generalize to ``$k$-sided pancakes''.

\subsection{Pancake Networks}

Interconnection networks connect single processors, or groups of processors, together.
In this context, the underlying graph is known as the network, and classic graph measurements (e.g. diameter, girth, connectivity) translate to different performance metrics.
Two networks related to pancake flipping are in Figure \ref{fig:graphs}.

The pancake network $\PGraph{n}$ was introduced in the 1980s \cite{orig-pancake} and various measurements were established (e.g. \cite{Heydari}).
Its vertex set is the set of permutations of $\{1,2,\ldots,n\}$ in one-line notation, which is denoted $\PERMS{n}$.
For example, $\PERMS{2} = \{12,21\}$.
There is an edge between permutations that differ by a \emph{prefix-reversal of length $\ell$}, which reverses the first $\ell$ symbols.
For example, $(3421,4321)$ is the $\ell=2$ edge between \stack[0.6]{3,4,2,1} and \stack[0.6]{4,3,2,1}.
Goodman's original problem is finding the maximum shortest path length to the identity permutation.
Since $\PGraph{n}$ is vertex-transitive, this value is simply its diameter.

The burnt pancake network $\BGraph{n}$ was introduced in the 1990s \cite{Cohen}.
Its vertex set is the set of signed permutations of $\{1,2,\ldots,n\}$, which is denoted $\BPERMS{n}$.
For example, $\BPERMS{2} = \{12,1\bar{2},\bar{1}2,\bar{1}\bar{2},21,2\bar{1},\bar{2}1,\bar{2}\bar{1}\}$ where overlines denote negative symbols.
There is an edge between signed permutations that differ by a \emph{sign-complementing prefix-reversal of length $\ell$}, which reverses the order and sign of the first $\ell$ symbols.
For example, $(\bar{2}\bar{1}34,1234)$ is the $\ell=2$ edge between \stackBurnt[0.6]{-2,-1,3,4} and \stackBurnt[0.6]{1,2,3,4}.

%The $k$-sided generalization is quite natual in this context.
%  (or \emph{$k$-coloured pancake network})
The \emph{$k$-sided pancake network} $\CGraph{n}{k}$ is a directed graph that was first studied in the 2000s \cite{Justan2002}. %by Justan, Muga II, and Sudborough \cite{Justan2002}.
Its vertex set is the set of $k$-coloured permutations of $\{1,2,\ldots,n\}$ in one-line notation, which is denoted $\CPERMS{n}{k}$. 
For example, $\CPERMS{2}{3}$ is illustrated below, where colours the $0$, $1$, $2$ are denoted using superscripts, or in {\bf black}, \rred{red}, \bblue{blue}.
\begin{align*} \label{eq:CPERMS23}
\CPERMS{2}{3} = \{ &
\C{1}{1}\C{2}{1},\C{1}{1}\C{2}{2},\C{1}{1}\C{2}{3},\C{1}{2}\C{2}{1},\C{1}{2}\C{2}{2},\C{1}{2}\C{2}{3},\C{1}{3}\C{2}{1},\C{1}{3}\C{2}{2},\C{1}{3}\C{2}{3},
\C{2}{1}\C{1}{1},\C{2}{1}\C{1}{2},\C{2}{1}\C{1}{3},\C{2}{2}\C{1}{1},\C{2}{2}\C{1}{2},\C{2}{2}\C{1}{3},\C{2}{3}\C{1}{1},\C{2}{3}\C{1}{2},\C{2}{3}\C{1}{3}
\} \\
= \{
& 1^0 2^0, 1^0 2^1, 1^0 2^2, 1^1 2^0, 1^1 2^1, 1^1 2^2, 1^2 2^0, 1^2 2^1, 1^2 2^2, 2^0 1^0, \ldots, 2^2 1^1, 2^2 1^2
% & 1^0 2^0, 1^0 2^1, 1^0 2^2, 1^1 2^0, 1^1 2^1, 1^1 2^2, 1^2 2^0, 1^2 2^1, 1^2 2^2, \\
% & 2^0 1^0, 2^0 1^1, 2^0 1^2, 2^1 1^0, 2^1 1^1, 2^1 1^2, 2^2 1^0, 2^2 1^1, 2^2 1^2
\}.
\end{align*}
There is a directed edge from $\pi_1 \in \CPERMS{n}{k}$ to $\pi_2 \in \CPERMS{n}{k}$ if $\pi_1$ can be transformed into $\pi_2$ by a \emph{colour-incrementing prefix-reversal of length $\ell$}, which reverses the order and increments the colour modulo $k$ of the first $\ell$ symbols.
For example, $(\C{2}{2} \C{1}{3} \C{3}{1} \C{4}{1}, \C{1}{1} \C{2}{3} \C{3}{1} \C{4}{1}) = (2^{1}1^{2}3^{0}4^{0}, 1^{0}2^{2}3^{0}4^{0})$ is a directed $\ell=2$ edge.
%burnt pancages in \stackBurnt[0.6]{-2,-1,3,4} to create \stackBurnt[0.6]{1,2,3,4}.

\begin{figure}
    \centering
    \begin{subfigure}[b]{0.48\columnwidth}
        \centering
        \includegraphics[width=0.85\textwidth]{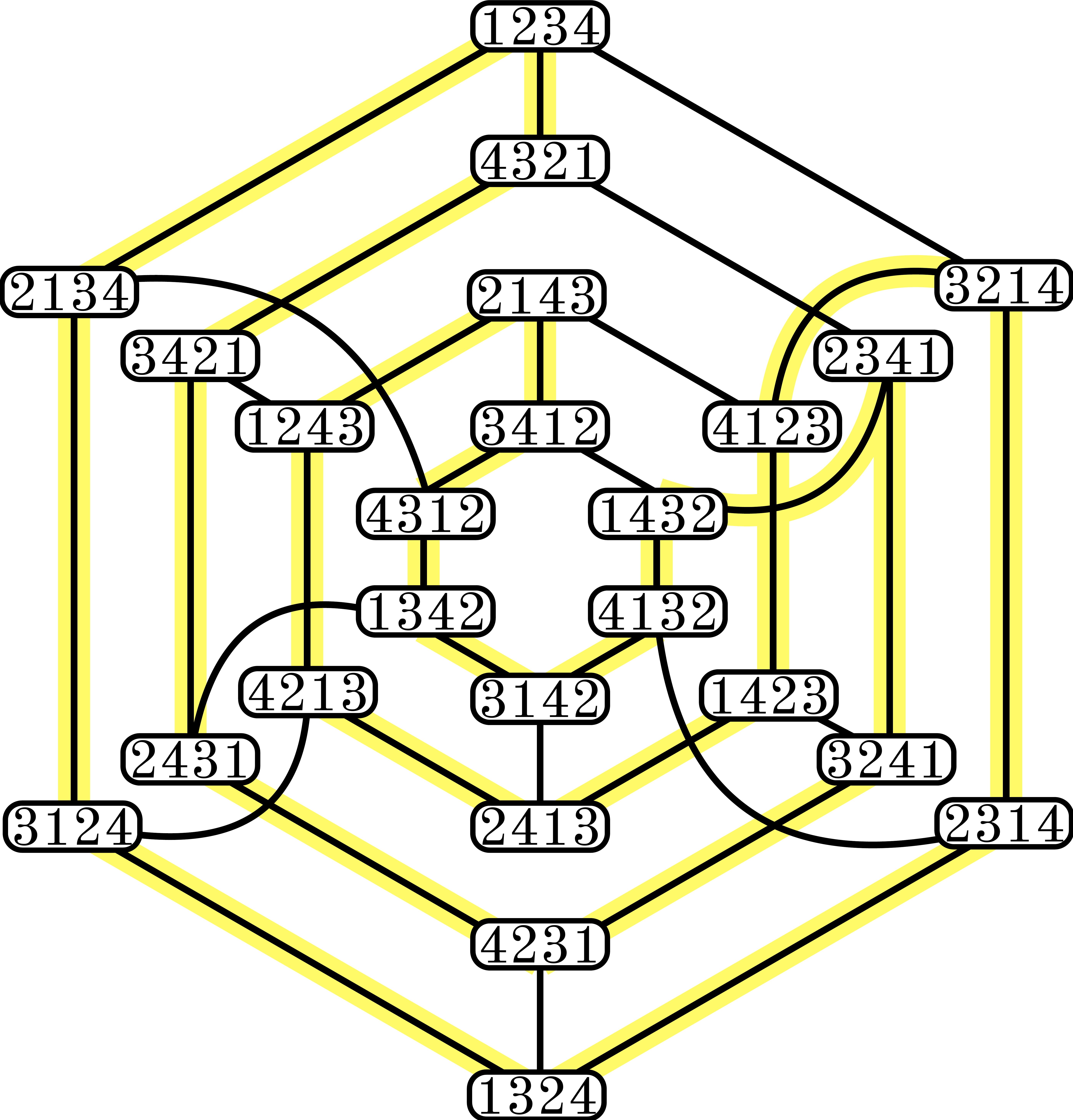} %{PancakeGraph4.pdf} %\hspace{0.1in}
        \caption{The pancake network $\PGraph{4}$.}
        \label{fig:graphs_PGraph4}
    \end{subfigure}
    %\hspace{0.01\columnwidth}
    \begin{subfigure}[b]{0.48\columnwidth}
        \centering
        \raisebox{1em}{\includegraphics[width=0.85\textwidth]{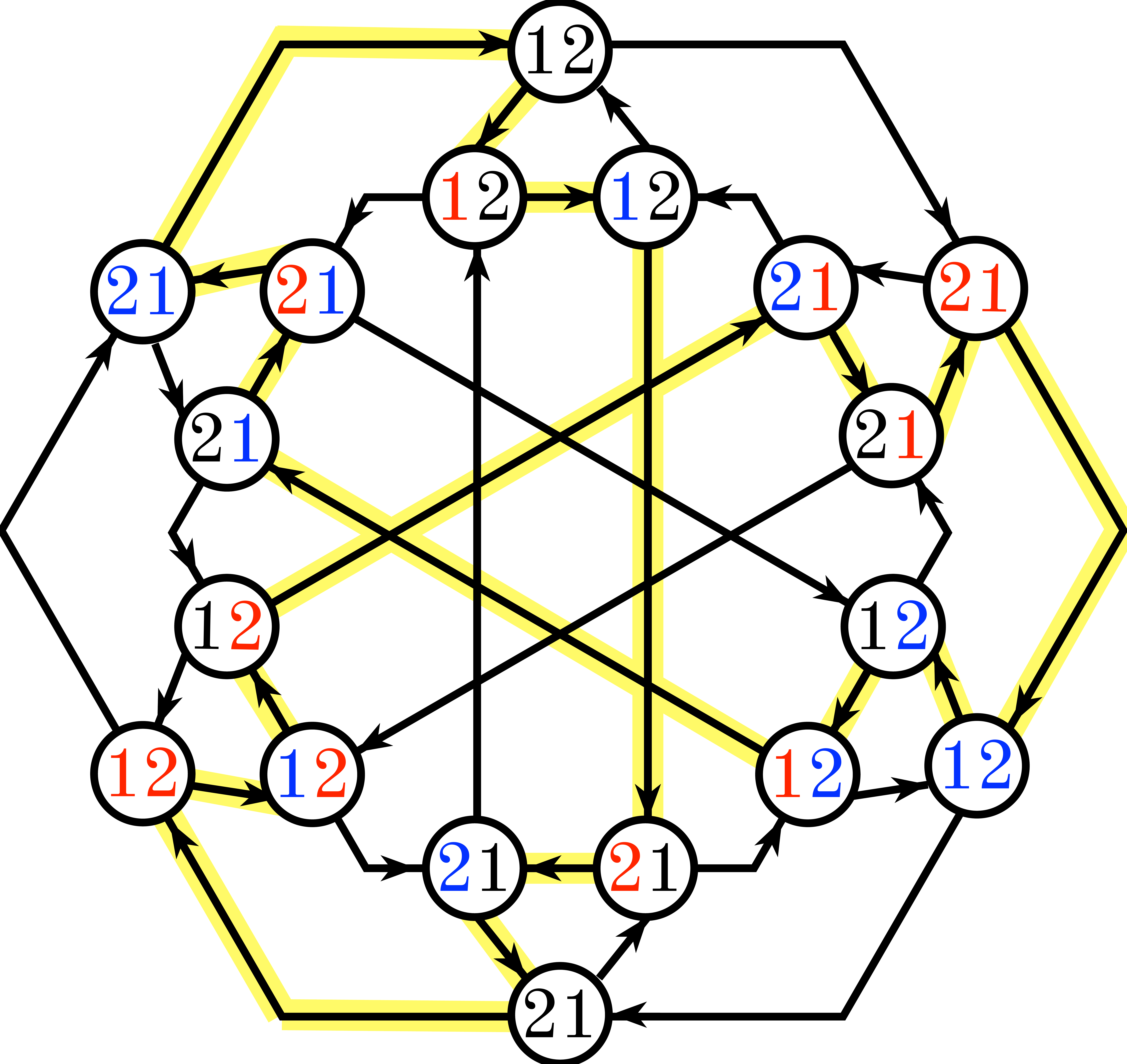}} %\hspace{0.1in}
        \caption{The $3$-sided pancake network $\CGraph{2}{3}$.}
        \label{fig:graphs_CGraph23}
    \end{subfigure}
    %\hspace{0.01\columnwidth}
    \caption{Hamilton cycles in a pancake network and a coloured pancake network.
    The highlighted cycles start at $1 2 \cdots n$ (or $1^0 2^0 \cdots n^0$) and are constructed by the greedy min-flip strategy.
    The colours $0,1,2$ in (b)  correspond to black, \red{red}, and \blue{blue}.}
    \label{fig:graphs}
\end{figure}

%Notice that coloured pancake networks generalize pancake and burnt pancake networks.
Notice that $\PGraph{n}$ and $\CGraph{n}{1}$ are isomorphic, while $\BGraph{n}$ and $\CGraph{n}{2}$ are isomorphic, so long as we view each undirected edge as two opposing directed edges.
% Less obvious is the isormophism we discussed by email ... increment by i colors where i is relatively prime with k.
It also bears mentioning that $\CGraph{n}{k}$ is a (connected) directed Cayley graph, and its underlying group is the wreath product of the cyclic group of order $k$ and the symmetric group of order~$n$.
%It is also relatively easy to prove that $\CGraph{n}{k}$ is connected.
Thus, $\CGraph{n}{k}$ is expected to have a directed Hamilton cycle by the Lov\'{a}sz conjecture. % [cite]

When the context is clear, or the distinction is not necessary, we use the term \emph{flip} for prefix-reversal (when $k=1$), sign-complementing prefix-reversal (when $k=2$), and colour-incrementing prefix-reversal (when $k>2$).

\subsection{(Greedy) Hamilton Cycles}

In this paper, we are not interested in shortest paths in pancake networks, but rather Hamilton cycles.
There are myriad ways that researchers attempt to build Hamilton cycles in highly-symmetric graphs, and the greedy approach is perhaps the simplest (see Williams \cite{GGA}).
This approach initializes a path at a specific vertex, then repeatedly extends the path by a single edge.
More specifically, it uses the highest priority edge (according to some criteria) that leads to a vertex that is not on the path.
The path stops growing when the current vertex is only adjacent to vertices on the path.
A Hamilton cycle has been found if every vertex is on the path, and there is an edge from the final vertex to the first vertex.
Despite its simplicity, the approach is known to work on many well-known graphs \cite{GGA}.

We show that the greedy approach works for the coloured pancake network $\CGraph{n}{k}$ when we prioritize the edges by shortest flip length.
More specifically, we start a path at $1^0 2^0 \cdots n^0 \in \CPERMS{n}{k}$, then repeatedly extend it to a new vertex along the edge that corresponds to the shortest colour-incrementing prefix-reversal.
We refer to this as the \emph{greedy min-flip construction}, denoted $\GreedyCMin{n}{k}$, and it is illustrated in Figure \ref{fig:graphs}.
When $k=1$, the cycle that we create is identical to the one given by Zaks \cite{zaks}, and when $k=2$, our cycle in the burnt pancake network was previously produced by Suzuki, N. Sawada, and Kaneko \cite{Kaneko}; however, both of these papers describe their cycles recursively.  The greedy construction of the cycles in the pancake and burnt pancake networks was previously given by J. Sawada and Williams \cite{greedyFun,greedyPancakes}.

\subsection{Combinatorial Generation}

Ostensibly, the primary contribution of this paper is the Hamiltonicity of $k$-sided pancake networks.
However, the authors' primary motivation was not in finding a Hamilton cycle, but rather in investigating its contributions to combinatorial generation.
\emph{Combinatorial generation} is the research area devoted to the efficient and clever generation of combinatorial objects.
By \emph{efficient} we mean that successive objects can be generated in amortized $O(1)$-time or worst-case $O(1)$-time, regardless of their size.
The former is known as \emph{constant amortized time (CAT)}, while the latter is known as \emph{loop-free}.
%To accomplish these goals, each successive object should be created from the previous object by some small modification.
By \emph{clever} we mean that non-lexicographic orders are often desirable.
%efficiency can be gained by generating the objects in an order that is not obvious.
% consider orders that can be more efficient than standard lexicographic orders.
When describing these alternate orders, the authors make liberal use of the term \emph{Gray code} --- in reference to the eponymous binary reflected Gray code patented by Frank Gray \cite{Gray:47}) --- and we refer to our Hamilton cycle as a \emph{colour-incrementing prefix-reversal Gray code} for coloured permutations.
Informally, it is a \emph{flip Gray code}.%, where \emph{flip} stands in for prefix-reversal (when $k=1$), and sign-complementing prefix-reversal (when $k=2$).

There are dozens of publications on the efficient generation of permutation Gray codes.
In fact, comprehensive discussions on this topic date back to Sedgewick's survey in 1977 \cite{Sedgewick:77}, with more modern coverage in Volume 4 of Knuth's \emph{The Art of Computer Programming} \cite{Knuth:V4}. % with new surveys coming out even today [cite].
However, to our knowledge, there are no published Gray codes for coloured permutations. This is surprising as the combinatorial \cite{Borodin1999,Chen2009,DuaneRemmel2011,Mansour2001,Mansour2003} and algebraic \cite{Athanasiadis2020,Bango2014,Shin2016} properties of coloured permutations have been of considerable interest. Work on the latter is due to the group theoretic interpretation of $\CPERMS{n}{k}$ as the wreath product of the cyclic and symmetric group, $\mathbb{Z}_k\wr S_n$.
%Besides being the first such order,
We find our new Grady code of interest for two additional reasons.
\begin{enumerate}
    \item Other greedy approaches for generating $\PERMS{n}$ do not seem to generalize to $\CPERMS{n}{k}$.
    \item Flips are natural and efficient operations in certain contexts.
    %\item Relationship to greedy Gray codes .. ie generalize previous results and give an interesting data point with regard to max flips.
\end{enumerate}

To expand on the first point, consider the Steinhaus-Johnson-Trotter (SJT) order of permutations, which dates back to the 1600s \cite{Knuth:V4}. % venerable
In this order, successive permutations differ by an \emph{adjacent-transition} (or \emph{swap}) meaning that adjacent values in the permutations change place.
In other words, the order for $\PERMS{n}$ traces a Hamilton path in the permutohedron of order $n$.
For example, SJT order for $n=4$ appears below
\begin{align*}
    12\underline{3\mathbf{4}},
    1\underline{2\mathbf{4}}3,
    \underline{1\mathbf{4}}23,
    41\underline{2\mathbf{3}},
    \underline{\mathbf{4}1}32,
    1\underline{\mathbf{4}3}2,
    13\underline{\mathbf{4}2},
    \underline{1\mathbf{3}}24,
    31\underline{2\mathbf{4}},
    3\underline{1\mathbf{4}}2,
    \underline{3\mathbf{4}}12,
    43\underline{1\mathbf{2}}, \\
    \underline{\mathbf{4}3}21,
    3\underline{\mathbf{4}2}1,
    32\underline{\mathbf{4}1},
    \underline{\mathbf{3}2}14,
    23\underline{1\mathbf{4}},
    2\underline{3\mathbf{4}}1,
    \underline{2\mathbf{4}}31,
    42\underline{\mathbf{3}1},
    \underline{\mathbf{4}2}13,
    2\underline{\mathbf{4}1}3,
    21\underline{\mathbf{4}3},
    \underline{\mathbf{2}1}34.
\end{align*}
The symbols that are swapped to create the next permutation are underlined, and the larger value is in bold.
The latter demarcation shows the order's underlying greedy priorities:
Swap the largest value.
For example, consider the fourth permutation in the list, 4123.
The largest value $4$ cannot be swapped to the left (since it is in the leftmost position) or the right (since $1423$ is already in the order), so the next option is to consider $3$, and it can only be swapped to the left, which gives the fifth permutation $4132$.
If this description is perhaps too brief, then we refer the reader to \cite{GGA}.

Now consider greedy generalizations of SJT to signed permutations.
The most natural generalization would involve the use of \emph{sign-complementing adjacent-transpositions} which swap and complement the sign of two adjacent values.
Unfortunately, any approach using these operations is doomed to fail.
This is because the operation does not change the parity of positive and negative values.
The authors experimented with other types of signed swaps --- complementing the leftmost or rightmost value in the swap, or the larger or small value in the swap --- without success.

More surprising is the fact that our greedy min-flip strategy works for coloured permutations, but the analogous max-flip strategy does not.
For example, the max-flip strategy creates the following path in $\CGraph{2}{3}$ before getting stuck.
\begin{equation*}
1^0 2^0, \
2^1 1^1, \
1^2 2^2, \
2^0 1^0, \
1^1 2^1, \
2^2 1^2, \
2^0 1^2, \
1^0 2^1, \
2^2 1^1, \
1^2 2^0, \
2^1 1^0, \
1^1 2^2, \ \text{\Lightning}
\end{equation*}
The issue is that the neighbors of last coloured permutation in the path are already on the path.
More specifically, a flip of length one transforms $1^1 2^2$ into $1^2 2^2$, and a flip of length two transforms $1^1 2^2$ into $2^0 1^2$, both of which appear earlier.
The failure of the max-flip strategy on coloured permutations is surprising due to the fact that it works for both permutations and signed-permutations \cite{greedyFun,greedyPancakes}.

To expand on the second point, note that the time required to flip a prefix is proportional its length.
In particular, if a permutation over $\{1,2,\ldots,n\}$ is stored in an array or linked list of length $n$, then it takes $O(m)$-time to flip a prefix of length $m$%
\footnote{Some unusual data structures can support flips of any lengths in constant-time \cite{Boustro}.}.
Our min-flip strategy ensures that the shortest possible flips are used.
In fact, the average flip length used in our Gray codes is bounded by $e = 2.71828\cdots$ when $k=1$, and the average is even smaller for $k>1$.
% when $k=1$, and is considerably smaller for $k>1$ colors.
%In other words, our flips are extremely short, and efficient, on average.

%Besides their use in genomics and interconnection networks,
We also note that flips can be the most efficient operation in certain situations.
For example, consider a brute force approach to the undirected travelling salesman problem, wherein every Hamilton path of the $n$ cities is represented by a permutation in $\PERMS{n}$.
If we iterate over the permutations using a prefix-reversal Gray code, then successive Hamilton paths differ in a single edge.
For example, the edges in $12345678$ and $43215678$ are identical, except that the former includes $(4,5)$ while the latter includes $(1,5)$.
Thus, the cost of each Hamilton cycle can be updated from permutation to permutation using one addition and subtraction.
%(In contrast, adjacent-transpositions change two edges, unless they involve the first or last value.) % while lexicographic orders can change up to 3?
More generally, flip Gray codes are the most efficient choice when the cost (or value) of each permutation depends on its unordered pairs of adjacent symbols.
Similarly, our generalization will be the most efficient choice when the cost (or value) of each coloured permutation depends on its unordered pairs of adjacent symbols \emph{and} the minimum distance between their colours.
% changing a single color alters two pairs

% Third point.
% Greedy easy to describe and recursive gets technical.
% Interested in developing general understanding
% Max doesn't work is a good data point
% Generalize Zaks and previous.

\subsection{New Results}

We present a flip Gray code for $\CPERMS{n}{k}$ that corresponds to a  Hamilton cycle in the $k$-sided pancake network.  We present  the following four different combinatorial algorithms for traversing the Hamilton cycle, each having unique and interesting properties:
\begin{enumerate} 
\item A greedy algorithm that is easy to describe, but requires an exponential amount of memory.
\item A recursive algorithm, that reveals the structure of the listing and can be implemented in $O(1)$-amortized time.
\item A simple successor rule approach that allows the cycle to start from any vertex (coloured permutation) and takes on average $O(1)$-time amortized over the entire listing.
\item A loop-free algorithm to generate the flip sequence iteratively.
\end{enumerate}

Before we present these algorithms in Section~\ref{sec:algos}, we first present some notation in Section~\ref{sec:back}.  We conclude with a summary and  related work in Section~\ref{sec:summary}.

\section{Notation} \label{sec:back}

%Formally, a \emph{coloured permutation} is a permutation where each element is assigned one of $k$ colours.  
Let $\pii = p_1p_2 \cdots p_n$ be a coloured permutation where each $p_i = v_i^{c_i}$ has value $v_i \in \{1,2,\ldots , n\}$ and colour $c_i \in \{0,1,\ldots , k-1\}$.  
Recall that $\CPERMS{n}{k}$ denotes the set of $k$-coloured permutations of $\{1,2,\ldots,n\}$.   
Observe that $ \CPERMS{n}{1}$ corresponds to regular permutations  and $\CPERMS{n}{2}$ corresponds to signed permutations.  
For the remainder of this paper, it is assumed that all permutations are coloured. 
%As an example: 
%$$\CPERMS{2}{2} = \{ 1^0 2^0, 1^0 2^1, 1^1 2^0, 1^1 2^1, 2^0 1^0, 2^0 1^1, 2^1 1^0, 2^1 1^1\} .$$

As mentioned earlier, a flip of a  permutation $\pii$, denoted $\flip{i}(\pii)$,  applies a prefix-reversal of length $i$ on $\pii$ that also increments the colour of the flipped elements by $1$ (modulo $k$). As an example for $k=3$:
$$\flip{4}(7^0 1^2 6^1 5^0 3^1 4^1 2^1)  = \mathbf{5^1 6^2 1^0 7^1} 3^1 4^1 2^1. $$

A \textit{pre-perm} is any prefix of a  permutation in $\CPERMS{n}{k}$, i.e. $\mathbf{p}=p_1p_2\cdots p_j$ is a  pre-perm if there exist $p_{j+1},\ldots p_{n}$ such that $p_1p_2\cdots p_n$ is a permutation. 
Note that if $j=n$, then the  pre-perm is a  permutation. 
Let  $\mathbf{p} = p_1p_2\cdots p_j$ be an arbitrary  pre-perm for given a $k$. 
For a given element $p_i=v_i^{c}$, let $p_i^{+s}=v_i^{(c_{i}+s)\  (\operatorname{mod}\ k)}$.  
For $0 \leq i < k$ , let $\mathbf{p^{+i}}$ denote  $p_1^{+i}p_2^{+i}\cdots p_j^{+i}$, i.e. $\mathbf{p}$ with the colour of each symbol $p_i$ incremented by $i$  modulo $k$.  
Note, $\mathbf{p^{+0}} = \mathbf{p}$.  Furthermore, let $\rho(\mathbf{p}) = \mathbf{p^{+(k-1)}} \cdot \mathbf{p^{+(k-2)}} \cdots \mathbf{p^{+0}} = r_1r_2\cdots r_{m}$  be a circular string of length $m=kj$ where $\cdot$ denotes the concatenation of symbols. %For example:  $p_1p_2 p_3 \cdot n = p_1p_2p_3n$..  
Let $\mathbf{\rho(p)_{i}}$ denote the length $j{-}1$ subword ending with $r_{i-1}$.   

\begin{exam}
%Let $k=4$, $n=2$.  If $\mathbf{p} = 1^0 2^0$, then 
%\[ \rho(\mathbf{p})=   1^3 2^3   \cdot   1^2 2^2  \cdot    1^1 2^1  \cdot   1^0 2^0  \mbox{ \ \ and \ \ }  \mathbf{\rho(p)_4} = 1^2.\] 
%
Consider a pre-perm $\mathbf{p} = 1^0 2^0 3^2$ where $j=3$ and  $k=4$.  Then 
\[\rho(\mathbf{p}) =  1^3 2^3 3^1  \cdot  1^2 2^2 3^0  \cdot   1^1 2^1 3^3 \cdot   1^0 2^0 3^2   \mbox{ \ \ and \ \ } \mathbf{\rho(p)_2} =  3^0 1^3.\]
%
%Let $k=5$, $n=5$.  If $\mathbf{p} = 5^2 3^4$, then 
%\[ \rho(\mathbf{p}) =  5^1 3^3  \cdot  5^0 3^2 \cdot   5^{4} 3^{1} \cdot 5^{3} 3^{0} \cdot  5^2 3^4  \mbox{ \ \ and \ \ }  \mathbf{\rho(p)_2} =  5^1.\]

\vspace{-0.1in}
\end{exam}

\noindent
For any  pre-perm $\mathbf{p}=p_1p_2\cdots p_j$, let $\rev{\mathbf{p}}$ denote the reverse of $\mathbf{p}$. i.e. $\rev{\mathbf{p}}=p_jp_{j-1}\cdots p_2p_1$.  
Note that $\rev{\mathbf{p}}$ is not equivalent to applying a flip of length $j$ to $\mathbf{p}$  when $k>1$ as the colours of each symbol do not change in $\rev{\mathbf{p}}$.  
For the remainder of this paper we will use $\mathbf{p}$ to denote a pre-perm, and when it is clear we will use $\pii$ to denote a permutation.

%\input{max-flips.tex}
%=====================================================================
\section{Constructions of a Cyclic Flip Gray code for $\CPERMS{n}{k}$ }\label{sec:algos}

In this section we present four different combinatorial algorithms for generating the same cyclic flip Gray code for  $\CPERMS{n}{k}$.  We begin by 
studying the listing generated by a greedy min-flip algorithm.   By studying the underlying recursive structure of the greedy listing, we provide a recursive description of the flip-sequence and prove it is equivalent to the flip-sequence generated by the greedy algorithm.  This proves that the greedy algorithm generates all permutations in $\CPERMS{n}{k}$.   We then present a successor-rule that determines the successor of a given permutation in the greedy min-flip listing in expected $O(1)$-time.  We conclude by showing how the flip-sequence can be generated via a loop-free algorithm.

%===========================
\subsection{Greedy Algorithm}
%Let  $\GreedyCMin{n}{k}  = \GREEDY( \CPERMS{n}{k}, 1^{0}2^{0}\cdots n^{0}, \langle \flip{1}, \flip{2}, \flip{3}, \ldots , \flip{n} \rangle )$, i.e. 

Recall that $\GreedyCMin{n}{k}$ denotes the greedy algorithm on $\CPERMS{n}{k}$ that starts at permutation $1^{0}2^{0}\cdots n^{0}$ and prioritizes the neighbors of each permutation in the $k$-sided pancake network by increasing flip length. 
%At each step the algorithm generates a new permutation by applying the smallest flip that results in a permutation not yet generated.    $\GreedyCMin{n}{k}$ terminates when no flips result in a new permutation.

%Our greedy algorithm for generating coloured permutations in $\CPERMS{n}{k}$ is denoted $\GreedyCMin{n}{k}$ and starts with the coloured permutation $1^{0}2^{0}\cdots n^{0}$ and, at each subsequent step, generates a new coloured permutation by applying the smallest available flip that results in a coloured permutation not yet generated. $\GreedyCMin{n}{k}$ terminates when no flips result in new coloured permutations. Unlike the max-flip algorithm, we will prove that this algorithm always exhaustively generates all permutation in $\CPERMS{n}{k}$ for all $n,k \ge 1$.  Moreover, the last permutation differs by a flip of length n from the first permutation, so the listing is a cyclic flip Gray code.  To prove this, we study the recursive properties of the listing in the next section
%
%We begin by looking at the greedy listing $\GreedyCMin{3}{3}$ which is an exhaustive listing of $\CPERMS{3}{3}$.  In this coloured example,
%black, red and blue correspond to the colours 0,1 and 2 respectively.

%
%---------------------------------
\begin{exam} \small  \label{exam:greedy}
The following listing (left of the vertical bar) denotes the output of $\GreedyCMin{3}{3}$ (read top to bottom, then left to right), where black, red and blue correspond to the colours 0,1 and 2 respectively.  This listing is exhaustive and cyclic;  the last permutation differs from the first permutation by a flip of length $n=3$.
To the right of the vertical line is the flip length required to get from the permutation in that position to its successor.
\begin{equation*}
\begin{array}{c @{\hspace{1em}} c @{\hspace{1em}} c @{\hspace{1em}} c@{\hspace{1em}} c @{\hspace{1em}} c @{\hspace{1em}} c @{\hspace{1em}} c @{\hspace{1em}} c @{\hspace{1em}}  | @{\hspace{1em}}  ccccccccc}
\C{1}{1} \C{2}{1} \C{3}{1} & \C{3}{2} \C{1}{1} \C{2}{1} &\C{2}{2} \C{3}{2} \C{1}{1} & \C{1}{2} \C{2}{2} \C{3}{2} &\C{3}{3} \C{1}{2} \C{2}{2} & \C{2}{3} \C{3}{3} \C{1}{2} & \C{1}{3} \C{2}{3} \C{3}{3} & \C{3}{1} \C{1}{3} \C{2}{3} & \C{2}{1} \C{3}{1} \C{1}{3} & 1 & 1& 1& 1& 1& 1& 1& 1& 1 \\
\C{1}{2} \C{2}{1} \C{3}{1} & \C{3}{3} \C{1}{1} \C{2}{1} &\C{2}{3} \C{3}{2} \C{1}{1} & \C{1}{3} \C{2}{2} \C{3}{2} &\C{3}{1} \C{1}{2} \C{2}{2} & \C{2}{1} \C{3}{3} \C{1}{2} & \C{1}{1} \C{2}{3} \C{3}{3} & \C{3}{2} \C{1}{3} \C{2}{3} & \C{2}{2} \C{3}{1} \C{1}{3} & 1 & 1& 1& 1& 1& 1& 1& 1& 1 \\
\C{1}{3} \C{2}{1} \C{3}{1} & \C{3}{1} \C{1}{1} \C{2}{1} &\C{2}{1} \C{3}{2} \C{1}{1} & \C{1}{1} \C{2}{2} \C{3}{2} &\C{3}{2} \C{1}{2} \C{2}{2} & \C{2}{2} \C{3}{3} \C{1}{2} & \C{1}{2} \C{2}{3} \C{3}{3} & \C{3}{3} \C{1}{3} \C{2}{3} & \C{2}{3} \C{3}{1} \C{1}{3} & 2 & 2 & 2 & 2 & 2 & 2 & 2 & 2 & 2  \\
\C{2}{2} \C{1}{1} \C{3}{1} & \C{1}{2} \C{3}{2} \C{2}{1} &\C{3}{3} \C{2}{2} \C{1}{1} & \C{2}{3} \C{1}{2} \C{3}{2} &\C{1}{3} \C{3}{3} \C{2}{2} & \C{3}{1} \C{2}{3} \C{1}{2} & \C{2}{1} \C{1}{3} \C{3}{3} & \C{1}{1} \C{3}{1} \C{2}{3} & \C{3}{2} \C{2}{1} \C{1}{3} & 1 & 1& 1& 1& 1& 1& 1& 1& 1\\
\C{2}{3} \C{1}{1} \C{3}{1} & \C{1}{3} \C{3}{2} \C{2}{1} &\C{3}{1} \C{2}{2} \C{1}{1} & \C{2}{1} \C{1}{2} \C{3}{2} &\C{1}{1} \C{3}{3} \C{2}{2} & \C{3}{2} \C{2}{3} \C{1}{2} & \C{2}{2} \C{1}{3} \C{3}{3} & \C{1}{2} \C{3}{1} \C{2}{3} & \C{3}{3} \C{2}{1} \C{1}{3} & 1 & 1& 1& 1& 1& 1& 1& 1& 1\\
\C{2}{1} \C{1}{1} \C{3}{1} & \C{1}{1} \C{3}{2} \C{2}{1} &\C{3}{2} \C{2}{2} \C{1}{1} & \C{2}{2} \C{1}{2} \C{3}{2} &\C{1}{2} \C{3}{3} \C{2}{2} & \C{3}{3} \C{2}{3} \C{1}{2} & \C{2}{3} \C{1}{3} \C{3}{3} & \C{1}{3} \C{3}{1} \C{2}{3} & \C{3}{1} \C{2}{1} \C{1}{3} & 2 & 2 & 2 & 2 & 2 & 2 & 2 & 2 & 2   \\
\C{1}{2} \C{2}{2} \C{3}{1} & \C{3}{3} \C{1}{2} \C{2}{1} &\C{2}{3} \C{3}{3} \C{1}{1} & \C{1}{3} \C{2}{3} \C{3}{2} &\C{3}{1} \C{1}{3} \C{2}{2} & \C{2}{1} \C{3}{1} \C{1}{2} & \C{1}{1} \C{2}{1} \C{3}{3} & \C{3}{2} \C{1}{1} \C{2}{3} & \C{2}{2} \C{3}{2} \C{1}{3}  & 1 & 1& 1& 1& 1& 1& 1& 1& 1 \\
\C{1}{3} \C{2}{2} \C{3}{1} & \C{3}{1} \C{1}{2} \C{2}{1} &\C{2}{1} \C{3}{3} \C{1}{1} & \C{1}{1} \C{2}{3} \C{3}{2} &\C{3}{2} \C{1}{3} \C{2}{2} & \C{2}{2} \C{3}{1} \C{1}{2} & \C{1}{2} \C{2}{1} \C{3}{3} & \C{3}{3} \C{1}{1} \C{2}{3} & \C{2}{3} \C{3}{2} \C{1}{3}  & 1 & 1& 1& 1& 1& 1& 1& 1& 1 \\
\C{1}{1} \C{2}{2} \C{3}{1} & \C{3}{2} \C{1}{2} \C{2}{1} &\C{2}{2} \C{3}{3} \C{1}{1} & \C{1}{2} \C{2}{3} \C{3}{2} &\C{3}{3} \C{1}{3} \C{2}{2} & \C{2}{3} \C{3}{1} \C{1}{2} & \C{1}{3} \C{2}{1} \C{3}{3} & \C{3}{1} \C{1}{1} \C{2}{3} & \C{2}{1} \C{3}{2} \C{1}{3} & 2 & 2 & 2 & 2 & 2 & 2 & 2 & 2 & 2  \\
\C{2}{3} \C{1}{2} \C{3}{1} & \C{1}{3} \C{3}{3} \C{2}{1} &\C{3}{1} \C{2}{3} \C{1}{1} & \C{2}{1} \C{1}{3} \C{3}{2} &\C{1}{1} \C{3}{1} \C{2}{2} & \C{3}{2} \C{2}{1} \C{1}{2} & \C{2}{2} \C{1}{1} \C{3}{3} & \C{1}{2} \C{3}{2} \C{2}{3} & \C{3}{3} \C{2}{2} \C{1}{3}  & 1 & 1& 1& 1& 1& 1& 1& 1& 1 \\
\C{2}{1} \C{1}{2} \C{3}{1} & \C{1}{1} \C{3}{3} \C{2}{1} &\C{3}{2} \C{2}{3} \C{1}{1} & \C{2}{2} \C{1}{3} \C{3}{2} &\C{1}{2} \C{3}{1} \C{2}{2} & \C{3}{3} \C{2}{1} \C{1}{2} & \C{2}{3} \C{1}{1} \C{3}{3} & \C{1}{3} \C{3}{2} \C{2}{3} & \C{3}{1} \C{2}{2} \C{1}{3}  & 1 & 1& 1& 1& 1& 1& 1& 1& 1 \\
\C{2}{2} \C{1}{2} \C{3}{1} & \C{1}{2} \C{3}{3} \C{2}{1} &\C{3}{3} \C{2}{3} \C{1}{1} & \C{2}{3} \C{1}{3} \C{3}{2} &\C{1}{3} \C{3}{1} \C{2}{2} & \C{3}{1} \C{2}{1} \C{1}{2} & \C{2}{1} \C{1}{1} \C{3}{3} & \C{1}{1} \C{3}{2} \C{2}{3} & \C{3}{2} \C{2}{2} \C{1}{3} & 2 & 2 & 2 & 2 & 2 & 2 & 2 & 2 & 2  \\
\C{1}{3} \C{2}{3} \C{3}{1} & \C{3}{1} \C{1}{3} \C{2}{1} &\C{2}{1} \C{3}{1} \C{1}{1} & \C{1}{1} \C{2}{1} \C{3}{2} &\C{3}{2} \C{1}{1} \C{2}{2} & \C{2}{2} \C{3}{2} \C{1}{2} & \C{1}{2} \C{2}{2} \C{3}{3} & \C{3}{3} \C{1}{2} \C{2}{3} & \C{2}{3} \C{3}{3} \C{1}{3}  & 1 & 1& 1& 1& 1& 1& 1& 1& 1 \\
\C{1}{1} \C{2}{3} \C{3}{1} & \C{3}{2} \C{1}{3} \C{2}{1} &\C{2}{2} \C{3}{1} \C{1}{1} & \C{1}{2} \C{2}{1} \C{3}{2} &\C{3}{3} \C{1}{1} \C{2}{2} & \C{2}{3} \C{3}{2} \C{1}{2} & \C{1}{3} \C{2}{2} \C{3}{3} & \C{3}{1} \C{1}{2} \C{2}{3} & \C{2}{1} \C{3}{3} \C{1}{3}  & 1 & 1& 1& 1& 1& 1& 1& 1& 1 \\
\C{1}{2} \C{2}{3} \C{3}{1} & \C{3}{3} \C{1}{3} \C{2}{1} &\C{2}{3} \C{3}{1} \C{1}{1} & \C{1}{3} \C{2}{1} \C{3}{2} &\C{3}{1} \C{1}{1} \C{2}{2} & \C{2}{1} \C{3}{2} \C{1}{2} & \C{1}{1} \C{2}{2} \C{3}{3} & \C{3}{2} \C{1}{2} \C{2}{3} & \C{2}{2} \C{3}{3} \C{1}{3} & 2 & 2 & 2 & 2 & 2 & 2 & 2 & 2 & 2  \\
\C{2}{1} \C{1}{3} \C{3}{1} & \C{1}{1} \C{3}{1} \C{2}{1} &\C{3}{2} \C{2}{1} \C{1}{1} & \C{2}{2} \C{1}{1} \C{3}{2} &\C{1}{2} \C{3}{2} \C{2}{2} & \C{3}{3} \C{2}{2} \C{1}{2} & \C{2}{3} \C{1}{2} \C{3}{3} & \C{1}{3} \C{3}{3} \C{2}{3} & \C{3}{1} \C{2}{3} \C{1}{3}   & 1 & 1& 1& 1& 1& 1& 1& 1& 1\\
\C{2}{2} \C{1}{3} \C{3}{1} & \C{1}{2} \C{3}{1} \C{2}{1} &\C{3}{3} \C{2}{1} \C{1}{1} & \C{2}{3} \C{1}{1} \C{3}{2} &\C{1}{3} \C{3}{2} \C{2}{2} & \C{3}{1} \C{2}{2} \C{1}{2} & \C{2}{1} \C{1}{2} \C{3}{3} & \C{1}{1} \C{3}{3} \C{2}{3} & \C{3}{2} \C{2}{3} \C{1}{3}  & 1 & 1& 1& 1& 1& 1& 1& 1& 1 \\
\C{2}{3} \C{1}{3} \C{3}{1} & \C{1}{3} \C{3}{1} \C{2}{1} &\C{3}{1} \C{2}{1} \C{1}{1} & \C{2}{1} \C{1}{1} \C{3}{2} &\C{1}{1} \C{3}{2} \C{2}{2} & \C{3}{2} \C{2}{2} \C{1}{2} & \C{2}{2} \C{1}{2} \C{3}{3} & \C{1}{2} \C{3}{3} \C{2}{3} & \C{3}{3} \C{2}{3} \C{1}{3} & 3 & 3 & 3 & 3 & 3 & 3 & 3 & 3 & \mathit{3}
\end{array}
\end{equation*}
Observe that each column of permutations ends with the same element.  Furthermore, the last permutation in each column is a subword of the cyclic word  $  \C{3}{1} \C{2}{1} \C{1}{1}  \C{3}{2} \C{2}{2} \C{1}{2} \C{3}{3} \C{2}{3} \C{1}{3}$.
\end{exam}

\begin{comment}
We give one more example using a more general notation that will be helpful to derive a recursive formula.
%
%---------------------------------
\begin{example} $\GreedyCMin{2}{4}$ (read down, then left to right):
\end{example}
%
\begin{center} 
{\renewcommand{\arraystretch}{1.0}
\begin{tabular}{*{12}{@{\ \ \ \ \ \ \  }c@{\,}}} 
$1^0 2^0$  &  $2^1 1^0$ &  $1^1 2^1$ &  $2^2 1^1$ & $1^2 2^2$  & $2^3 1^2$& $1^3 2^3$  & $2^0 1^3$ \\	
$1^1 2^0$  &  $2^2 1^0$ & $1^2 2^1$  & $2^3 1^1$ & $1^3 2^2$  & $2^0 1^2$& $1^0 2^3$  & $2^1 1^3$ \\
$1^2 2^0$  &  $2^3 1^0$ & $1^3 2^1$  & $2^0 1^1$ & $1^0 2^2$  & $2^1 1^2$ & $1^1 2^3$ & $2^2 1^3$   \\	
$1^3 2^0$  &  $2^0 1^0$ & $1^0 2^1$  & $2^1 1^1$ & $1^1 2^2$  & $2^2 1^2$ & $1^2 2^3$ & $2^3 1^3$  \\
\end{tabular}}
\end{center}

\end{comment}
%
%
%
%In particular, notice that that successive columns contain every permutation ending in $2^0, 1^0, 2^1, 1^1, 2^2, 1^2, 2^3, 1^3$ respectively.
%Furthermore, notice that the first string in each column is a length $n$ subword of the circular string $1^3 2^3 \cdot 1^2 2^2 \cdot 1^1 2^1 \cdot 1^0 2^0$. Similarly,
%the last string in each column is a  length $n$ subword of the circular string $2^01^0 \cdot 2^1 1^1 \cdot 2^21^2 \cdot 2^3 1^3$. It turns out that there is an underlying recursive structure of the greedy minimum flip strategy that is responsible for this well-structured listing. We will describe this recursion and use it to prove that $\GreedyCMin{n}{k}$ produces a cyclic flip gray code for all $n$ and $k$.

Unlike the max-flip approach, we will prove that $\GreedyCMin{n}{k}$  exhaustively generates all permutations in $\CPERMS{n}{k}$ for all $n,k \ge 1$.  We also show that the last permutation in the listing differs by a flip of length $n$ from the first permutation, so the listing is a cyclic flip Gray code. 
To prove this result,  we study the underlying recursive structure of the resulting listings and examine the flip-sequences.

%===========================
\subsection{Recursive Construction}
 
%A recursive formulation for coloured permutations is similar to  the formulation for the non-signed case with a minor change to some notation.  
%

By applying the two observations made following the listing of $\GreedyCMin{3}{3}$  in Example~\ref{exam:greedy},  we arrive at the following recursive definition for a given pre-perm $\mathbf{p}$ of a  permutation in $\CPERMS{n}{k}$:
\footnotesize
\begin{eqnarray}  \label{eq:bminflip}
 \RecCMin{\mathbf{p}}{k}    &= &  \RecCMin{\mathbf{\rho(p)_{m}}}{k} \cdot r_{m}, \  \RecCMin{\mathbf{\rho(p)_{m-1}}}{k} \cdot r_{m-1},   \ldots , \ \RecCMin{\mathbf{\rho(p)_1}}{k}  \cdot r_1,
 \end{eqnarray}
 \normalsize
where  $\RecCMin{p_1}{k}  = p_1^{+0}, p_1^{+1}, p_1^{+2},\ldots, p_1^{+(k-1)}$.
We will prove that $\RecCMin{1^02^0\cdots n^0}{k}$ generates the same (exhaustive) listing of permutations as $\GreedyCMin{n}{k}$.  

\begin{lemma}
\label{lem:BminFL}
Let $\mathbf{p}=p_1p_2p_3\cdots p_j$ be a  pre-perm of a  permutation in $\CPERMS{n}{k}$ for some $n\ge j$. Then the first and last  pre-perms in the listing $\RecCMin{\mathbf{p}}{k}$ are $\mathbf{p}$ and $\rev{\mathbf{p^{+(k-1)}}}$, respectively.   
\end{lemma}
%
%\begin{sloppypar}
\noindent
\begin{proof}
%Let $n$ and $k$ be given.
 The proof proceeds by induction on $j$.  When $j=1$, we have $\mathbf{p}=\rev{\mathbf{p}}=p_1$, so $\RecCMin{\mathbf{p}}{k} = \mathbf{p}, \mathbf{p^{+1}}$, $\mathbf{p^{+2}}, \ldots,$ $\mathbf{p^{+(k-1)}}$.   Since $\mathbf{p^{+(k-1)}}= \rev{\mathbf{p^{+(k-1)}}}$  the claim holds.  Now for $1 \leq j < n$ and any  pre-perm $\mathbf{p}=p_1p_2\cdots p_j$ of a  permutation in $\CPERMS{n}{k}$, suppose that the first and last  pre-perms in $\RecCMin{\mathbf{p}}{k}$ are $\mathbf{p}$ and  $\rev{\mathbf{p^{+(k-1)}}}$ respectively. Let $\mathbf{p}=p_1p_2\cdots p_j p_{j+1}$ be a  pre-perm of a  permutation in $\CPERMS{n}{k}$. By definition, the first  pre-perm of $\RecCMin{\mathbf{p}}{k}$  is the first  pre-perm of $\RecCMin{\mathbf{\rho(p)_{m}}}{k} \cdot r_{m}$ where $m=(j+1)k$.  By definition of $\mathbf{\rho(p)}$ and $\mathbf{\rho(p)_{m}}$, it is clear that $r_m=p_{j+1}$ and $\mathbf{\rho(p)_{m}}=p_1p_2\cdots p_{j-1}p_j$. Applying the inductive hypothesis, the first  pre-perm of $\RecCMin{p_1p_2\cdots p_{j-1}p_j}{k}$ is $p_1p_2\cdots p_{j-1}p_j$. Therefore, the first  pre-perm of $\RecCMin{\mathbf{p}}{k}$ is $p_1p_2\cdots p_{j-1}p_{j}\cdot p_{j+1}=\mathbf{p}$.
Similarly, the last  pre-perm of $\RecCMin{\mathbf{p}}{k}$  is the last  pre-perm of $\RecCMin{\mathbf{\rho(p)_1}}{k}  \cdot r_1$. Now,  $r_1=p_1^{+(k-1)}$ and $\mathbf{\rho(p)_1}= p_{2}p_3\cdots p_np_{n+1}$ and, by the inductive hypothesis, the last  pre-perm in $\RecCMin{\mathbf{\rho(p)_1}}{k}$  is $p_{n+1}^{+(k-1)}p_{n}^{+(k-1)}\cdots p_{2}^{+(k-1)}$. Therefore, the last  pre-perm of $\RecCMin{\mathbf{p}}{k}$ is $\rev{\mathbf{p^{+(k-1)}}}$. \hfill $\Box$
\end{proof}
%\end{sloppypar}
%===========================

Define the sequence $\SEQQ^k_{n}$ recursively as
\begin{equation} \label{eq:seq}
\SEQQ^k_{n} = 
\begin{cases}
1^{k-1} & \text{ if $n=1$} \\
  (\SEQQ^k_{n-1}, n)^{kn-1}, \SEQQ^k_{n-1} & \text{ if $n> 1$}. \\
\end{cases}  
\end{equation}

\noindent We will show that $\SEQQ^k_{n}$ is the flip-sequence for both $\RecCMin{\mathbf{p}}{k}$ and  $\GreedyCMin{n}{k}$. This flip-sequence is a straightforward generalization of the recurrences for non-coloured permutations~\cite{zaks} and signed permutations~\cite{greedyPancakes}.
\begin{lemma}
\label{lem:BminR}
%For $n \geq 1$, the list $\BMinF{\mathbf{p}}$  can be created by starting from $\mathbf{p}$  and  applying the sequence of flips in $\BminSEQQ_{n}$.
For $n \geq 1$ , $k\ge 1$, and $\pii \in\CPERMS{n}{k}$, the flip-sequence for  $\RecCMin{\pii}{k}$  is $\SEQQ^k_{n}$.

\end{lemma}
\begin{proof}
By induction on $n$.  In the base case $\RecCMin{p_1}{k} = p_1, p_1^{+1},p_1^{+2},\ldots, p_1^{+(k-1)}$ and the flip-sequence is $\SEQQ^k_{1} =  1^{k-1}$. For $n \geq 1$ assume that the 
sequence of flips used to create  $\RecCMin{p_1p_2p_3\cdots p_n}{k}$  is given by $\SEQQ^k_{n}$.  
Consider $\RecCMin{\pi}{k}$ where $\pi = p_1p_2p_3\cdots p_{n+1}\in\CPERMS{n+1}{k}$. By the inductive hypothesis, it suffices to show that the last  permutation of  $\RecCMin{\rho(\pi)_{i}}{k} \cdot r_{i}$ and the first  permutation of $\RecCMin{\rho(\pi)_{i-1}}{k} \cdot r_{i-1}$ differ by a flip of length $n+1$ for $i= 2,3,\ldots,m\ (=k(n+1))$. By definition, $\rho(\pi)_{i}=r_{i-n}r_{i-(n-1)}\cdots r_{i-2}r_{i-1}$ where the indices are taken modulo $m$. Therefore, by Lemma~\ref{lem:BminFL}, the last  permutation in $\RecCMin{\rho(\pi)_{i}}{k}$ is $(r_{i-1}r_{i-2}\cdots r_{i-(n-1)}r_{i-n})^{+(k-1)}$. Applying a flip of length $n+1$ to $\RecCMin{\rho(\pi)_{i}}{k}\cdot r_i$ yields
\begin{align}
r_i^{+1}r_{i-n}r_{i-(n-1)}\cdots r_{i-2}r_{i-1}. \label{wrd:reci}
\end{align}
By   Lemma~\ref{lem:BminFL}, the first  permutation of $\RecCMin{\rho(\pi)_{i-1}}{k}$ is $r_{i-(n+1)}r_{i-n}\cdots r_{i-3}r_{i-2}$.  By the definition of $\rho(\pi)$, it follows that $r_{i-(n+1)}=r_i^{+1}$. Thus, from (\ref{wrd:reci}), it follows that $\RecCMin{\rho(\pi)_{i}}{k} \cdot r_{i}$ and the first  permutation of  $\RecCMin{\rho(\pi)_{i-1}}{k} \cdot r_{i-1}$ differ by a flip of length $n+1$. By applying the inductive hypothesis, the flip-sequence for $\RecCMin{\pi }{k}$ is  $ (\SEQQ^k_{n}, n+1)^{k(n+1)-1}, \SEQQ^k_{n}$ which is exactly $\SEQQ^k_{n+1}$. \hfill $\Box$
\end{proof}
%===========================

%
\begin{theorem}\label{thm:BminGray}
For $n \geq 1$, $k\ge 1$, and $\pii \in\CPERMS{n}{k}$, $\RecCMin{\pii}{k}$  is a cyclic flip Gray code for  $\CPERMS{n}{k}$, where the  first and last  permutations differ by a flip of length $n$.
\end{theorem}
\begin{proof}
From Lemma~\ref{lem:BminR}, the flip-sequence for  $\RecCMin{\pii}{k}$ is given by $\SEQQ^k_{n}$.
Inductively, it is easy to see that the length of the flip-sequence $\SEQQ^k_n$ is $k^nn!-1$ and that each  permutation of $\RecCMin{\pii}{k}$ is unique.  Thus, each of the $k^nn!$  permutations must be listed exactly once and, from Lemma~\ref{lem:BminFL}, the first and last  permutations of the listing differ by a flip of length $n$, making $\RecCMin{\pii}{k}$ a cyclic flip Gray code for  permutations. \hfill $\Box$

\end{proof}
%===========================
%Lemma may be added if we find the point about perms with given suffix not convincing in proof of lem:BminG

%\begin{lemma}
%\label{lem:suffixflip}
%Let $n\ge 1$, $k\ge 1$. If $\mathbf{p}=p_1p_2\cdots p_n$ and $\mathbf{q}=q_1q_2\cdots q_n$ are  permutations, then in the listing $\RecCMin{\mathbf{p}}{k}$, every  permutation with suffix $S=q_tq_{t+1}\cdots q_n$ appears before a  permutation with that suffix is flipped by a flip of length $t$.  
%\end{lemma}
%
\begin{lemma}
\label{lem:BminG}
For $n \geq 1$ and $k\ge 1$, the flip-sequence for $\GreedyCMin{n}{k}$  is  $\SEQQ^k_{n}$.
\end{lemma}
\begin{proof}  
By contradiction. Suppose the sequence of flips used by $\GreedyCMin{n}{k}$ differs from $\SEQQ^k_{n}$ and let $j$ be the smallest value such that the $j$-th flip used to create $\GreedyCMin{n}{k}$ differs from the $j$-th value of $\SEQQ^k_{n}$.  Let these flip lengths be $s$ and $t$ respectively.   Since $\GreedyCMin{n}{k}$  follows a greedy minimum-flip strategy and because $\SEQQ^k_{n}$ produces a valid flip Gray code for  permutations by Theorem~\ref{thm:BminGray} where no  permutation is repeated, it must be that $s < t$.   Let $\pii = p_1p_2p_3\cdots p_n$ denote the $j$-th  permutation in the listing $\GreedyCMin{n}{k}$, i.e. the  permutation immediately prior to the $j$-th flip. 
Since $\sigma^k_n$ is the flip-sequence for $\RecCMin{1^02^0\cdots n^0}{k}$ by Lemma~\ref{lem:BminR}, from the recursive definition it follows inductively that all other permutations with suffix $p_tp_{t+1}\cdots p_n$ appear before $\pii$ in $\RecCMin{1^02^0\cdots n^0}{k}$, since no permutations are repeated by Theorem~\ref{thm:BminGray}. Since $\sigma^k_n$ and the sequence of flips used by $\GreedyCMin{n}{k}$ agree until the $j$-th value, all other  permutations with suffix $p_tp_{t+1}\cdots p_n$ appear before $\pii$ in $\GreedyCMin{n}{k}$. Therefore, flipping $\pii$ by a flip of length $s<t$ results in a  permutation already visited in  $\GreedyCMin{n}{k}$ before index $j$  contradicting the fact that $\GreedyCMin{n}{k}$ produces a list of  permutations without repetition.   
\hfill $\Box$
\end{proof}
%===========================
%

By definition, $\GreedyCMin{n}{k}$  starts with the  permutation $1^02^0\cdots n^0$ and by Lemma~\ref{lem:BminFL},\\
$\RecCMin{1^02^0\cdots n^0}{k}$ also starts with $1^02^0\cdots n^0$.  Since they are each created by the  same flip-sequence by Lemma~\ref{lem:BminR} and Lemma~\ref{lem:BminG}, we get the following corollary. 

\begin{corollary}  \label{cor:Bmin}
For $n \geq 1$ and $k\ge 1$, the listings $\GreedyCMin{n}{k}$ and  $\RecCMin{1^02^0\cdots n^0}{k}$ are equivalent.
\end{corollary}  
%

%===============

%==================
\subsection{Successor Rule}
\label{sec:signedminsuccess}

In this section, we will generalize the successor rules found for non-coloured permutations and signed permutations in \cite{successor} for $\GreedyCMin{n}{k}$ for $k> 2$. We say a  permutation in $\CPERMS{n}{k}$ is \emph{increasing} if it corresponds to a length $n$ subword of the circular string $\rho(1^02^0\cdots n^0)$.  For example if $n=6$ and $k=4$, then the following permutations are all increasing: 
$$2^33^3  4^3 5^3 6^3 1^2 \ \ \ \  \  5^16^11^02^03^04^0 \ \ \ \ \   1^02^03^04^05^06^0 \ \ \ \  \  5^06^01^32^33^34^3.$$  A  permutation is \emph{decreasing} if it is a reversal of an increasing permutation.  
%An   $r$-permutation is any length $r$ prefix of a  permutation.
A pre-perm is \emph{increasing} (\emph{decreasing}) if it corresponds to a  subsequence of an increasing (decreasing)   permutation (when the permutation is thought of as a sequence).  For example, $5^16^12^04^0$ is an increasing pre-perm, but
$5^1 2^0 4^0 6^0$ and $1^2 2^2 3^1 4^0$ are not. Given a  permutation $\pii_2$, let $\suc{\pii_2}$ denote the successor of $\pii_2$ in  $\RecCMin{\pii}{k}$ when the listing is considered to be cyclic.

\begin{lemma} 
\label{lem:Bminsucc}
Let $\pii_2 = q_1q_2\cdots q_n$ be a  permutation in the (cyclic) listing $\RecCMin{\pii}{k}$, where $\pii=p_1p_2\cdots p_n$ is increasing.   Let $q_1 q_2 \cdots q_j$ be the longest prefix of $\pii_2$ that is decreasing.  Then $\suc{\pii_2} = \flipper{\pii_2}{j}.$

\end{lemma}
\begin{proof}
%We prove a slightly stronger result: Equation \eqref{eq:minsucc} holds when any increasing $\pii = p_1 p_2 \cdots p_n$ is substituted for $1 \, 2 \, \cdots \, n$.
By induction on $n$. When $n=1$, the result follows trivially as only flips of length $1$ can be applied. Now, for $n>1$, we focus on the  permutations whose successor is the result of a flip of length $n$ and the result will follow inductively by the recursive definition of $\RecCMin{\pii}{k}$.
By Lemma~\ref{lem:BminR}, the successor of $\pii_2$ will
be $\flipper{\pii_2}{n}$ if and only if it is the last  permutation in one of the recursive listings of the form $\RecCMin{\rho(\pii)_{\mathbf{i}}}{k} \cdot r_i$. Recall that $r_i$ is the $i$-th element in $\rho(\pii)$ when indexed from $r_1=p_1^{+(k-1)}$ to $r_m=p_n$. As it is clear that at most one  permutation is decreasing in each recursive listing, it suffices to show that the last  permutation in each listing is decreasing to prove the successor rule holds for flips of length $n$. By Lemma~\ref{lem:BminFL}, the last permutation in $\RecCMin{\rho(\pii)_{\mathbf{i}}}{k}\cdot r_i$ is $\rev{\mathbf{s}}\cdot r_i$ where $\mathbf{s}=\rho(\pii)_{\mathbf{i}}^{\mathbf{+(k-1)}}$. Since $\pii$ is increasing, it is clear that $\rho(\pii)_{\mathbf{i}}$ is increasing and therefore that $\mathbf{s}$ is increasing. Hence, $\rev{\mathbf{s}}$ is decreasing by definition. Furthermore, by the definition of the circular word $\rho(\pii)$, the element immediately before $r_{i-1-(n-1)}^{+(k-1)}$ in $\rho(\pii)$ is $r_i$ (note the subscript $i-1-(n-1)$ is considered modulo $nk$ here). Therefore,  $\rev{\mathbf{s}}\cdot r_i$ is decreasing. Therefore, the successor rule holds for flips of length $n$ and thus for flips of all lengths by induction. \hfill $\Box$
\end{proof}
\noindent
\begin{exam}
With respect to the listing $\RecCMin{1^0 2^0 3^0 4^0 5^0 6^0}{10}$, $$\suc{3^8 2^8 5^9 4^9 1^7 6^3}=\flip{4}(3^8 2^8 5^9 4^9 1^7 6^3)=4^0 5^0 2^9 3^9  1^7 6^3$$ and $$\suc{1^83^72^65^54^36^2}=\flip{1}(1^83^72^65^54^36^2)=1^93^72^65^54^36^2.$$ 
\end{exam}
%\noindent
%\red{ UPDATED EXAMPLE REQUIRED
%As an example, consider the permutation  3764512 with respect to the listing $\RecMin{12\cdots n}$.  The prefix 3764 is the longest one that is decreasing, thus $j=4$ and the
%next permutation in the listing is $\flip{4}(3764512)$.
%}

\noindent
By applying the previous lemma,  computing $\suc{\pii_2}$ for a permutation in the listing  $\RecCMin{\pii}{k}$
can easily be done in $O(n)$-time as described in the pseudocode given in Algorithm~\ref{alg:minsuccessor}.
%Determining the value $j$ in this successor rule can easily be determined in $O(n)$ time by applying the pseudocode  given in Algorithm~\ref{alg:minsuccessor}. 

%-----------------------
\begin{algorithm}[h]
  \caption{Computing the successor of $\pii$ in the listing $\RecCMin{1^02^0\cdots n^0}{k}$}
  \label{alg:minsuccessor}
  \small
  \begin{algorithmic}[1]

\Function{Successor}{$\pii$}  

	\State $incr \gets 0$ 
	\For{ $j\gets 1$ {\bf to} $n-1$} 
		\If { $v_j < v_{j+1}$ } \  $incr \gets incr+1$  \EndIf
		\If  {$incr = 2$ {\bf or} ($incr = 1$ {\bf and} $v_{j+1} < v_1$)}  \  \Return $\flip{j}(\pii)$  \EndIf
		%\If  {$|p_j| < |p_{j+1}|$ {\bf and} \Call{sign}{$p_j$} $=$ \Call{sign}{$p_{j+1}$} }  \  \Return $j$  \EndIf
		%\If  {$|p_j| > |p_{j+1}|$ {\bf and} \Call{sign}{$p_j$} $\neq$ \Call{sign}{$p_{j+1}$}   }  \  \Return $j$  \EndIf

		%\If  {$|p_j| < |p_{j+1}|$ {\bf and} ( ($c_{j+1} - c_j + k) \bmod k \neq 1$)}  \  \Return $\Bflip{j}(\pii)$  \EndIf
		%\If  {$|p_j| > |p_{j+1}|$ {\bf and} ($c_{j+1} \neq c_j$)}  \  \Return $\Bflip{j}(\pii)$  \EndIf
		
		\If  {$k > 1$ {\bf and} $v_j < v_{j+1}$ {\bf and} ( ($c_{j+1} - c_j + k) \bmod k \neq 1$)}  \  \Return $\flip{j}(\pii)$  \EndIf
		\If  {$k > 1$ {\bf and} $v_j > v_{j+1}$ {\bf and} (  $c_{j} \neq c_{j+1}$)}  \  \Return $\flip{j}(\pii)$  \EndIf

	\EndFor
	\State \Return $\flip{n}(\pi)$
\EndFunction
    
  \end{algorithmic}
\end{algorithm}
%-----------------------

\begin{theorem}
{\sc Successor}($\pii$) returns the length of the flip required to obtain the successor of  $\pii$ in the listing $\RecCMin{1^02^0\cdots n^0}{k}$ in $O(n)$-time.
\end{theorem}

Though the worst case performance of  {\sc Successor}($\pii$) is $O(n)$-time, on average it is much better.
Let $\overline{\SEQQ}_{n}^{k}$ denote $(\SEQQ_n^{k},n)$, i.e. the sequence of flips used to generate the listing $\RecCMin{\pii}{k}$ with an extra flip of length $n$ at the end to return to the starting  permutation. Our goal is to determine the average flip length of $\overline{\SEQQ}_{n}^{k}$, denoted $avg(n,k)$. Note that our analysis generalize the results for $avg(n,1)$ \cite{zaks} and $avg(n,2)$ \cite{greedyPancakes}.  

\begin{lemma}\label{lem:avg}
For $n\ge 1$ and $k\ge 1$, $$avg(n,k)  =    \sum_{j=0}^{n-1}  \frac{1}{k^{j}  j!}.$$ Moreover, $avg(n,k)<\sqrt[k]{e}$.
\end{lemma}
\begin{proof}
By definition of $\SEQQ_{n}^k$, it is not difficult to see that $\overline{\SEQQ}_{n+1}^k$ is equivalent to the concatenation of $(n+1)k$ copies of $\overline{\SEQQ}_{n}^k$ with  the last element in every copy of $\overline{\SEQQ}_{n}^k$ incremented by $1$. Therefore, we have 

\begin{align*}
avg(n+1,k)&=\frac{\left(1+\sum\limits_{f\in \overline{\SEQQ}_n^k}f\right)(n+1)k } {(n+1)!k^{n+1}}\\
&= \frac{\sum\limits_{f\in \overline{\SEQQ}_n^k}f } {n!k^{n}}+\frac{1} {n!k^{n}}\\
&= avg(n,k)+\frac{1} {n!k^{n}}.
\end{align*}
\noindent Hence, with the trivial base case that $avg(1,k)=1$, we have

%To determine the average flip length at each step, a similar analysis as the signed case can be applied to  permutations where we replace each 2 with a $k$.  Thus:
%
%\small
\begin{eqnarray*} 
avg(n,k) & = &   \sum_{j=0}^{n-1}  \frac{1}{k^{j}  j!}.
\end{eqnarray*}
%\normalsize
%
Therefore, $$avg(n,k)<\sum_{j=0}^{\infty}  \frac{1}{k^{j}  j!}=\sqrt[k]{e}$$
\noindent by applying the well-known Maclaurin series expansion for $e^{x}$. \hfill $\Box$
\end{proof}

\begin{comment}
%-----------------------
\begin{algorithm}[h]
  \caption{Exhaustive algorithm to list the ordering $\RecCMin{1^02^0\cdots n^0}{k}$ of $\CPERMS{n}{k}$}
  \label{alg:repeatedCMin}
  \small
  \begin{algorithmic}[1]

\Procedure{Gen}{}  

	\State $\pii \gets 1^02^0\cdots n^0$
	\Repeat	
		\State \Call{Visit}{$\pii$}
		\State $j \gets \Call{Successor}{\pii}$
		\State $\pii \gets \flip{j}(\pii)$

	\Until{$j =n$ {\bf and} $p_1 = 1$} 

	%\EndWhile		
\EndProcedure

  \end{algorithmic}
\end{algorithm}
\end{comment}

\noindent
Observe that the {\sc Successor} function runs in expected $O(1)$-time when the permutation is passed by reference because the average flip length is bounded above by the constant $\sqrt[k]{e}$ as proved in Lemma~\ref{lem:avg}.  Thus, by repeatedly applying the successor rule, we obtain a CAT algorithm for generating  $\RecCMin{1^02^0\cdots n^0}{k}$.
%Pseudocode is given in Algorithm~\ref{alg:repeatedCMin}.
%Observe that the algorithm returns to the initial permutation when $j=n$ and $p_1 = 1$.

%\begin{theorem}
%{\sc Gen} produces the ordering $\RecCMin{1^02^0\cdots n^0}{k}$  of $\CPERMS{n}{k}$ in $O(1)$-amortized time.
%\end{theorem}

%===============
%===============
\subsection{Loop-free Generation of the Flip-Sequence $\SEQQ^k_{n}$}

Based on the recursive definition of the flip-sequence  $\SEQQ^k_{n}$ given in \eqref{eq:seq},  Algorithm~\ref{alg:iter} will generate $\SEQQ^k_{n}$ in a loop-free manner.  The algorithm generalizes a similar algorithm presented by Zaks for non-coloured permutations~\cite{zaks}.  The next flip length $x$ is computed using an array of counters $c_1,c_2,\ldots ,c_{n+1}$ initialized to 0, and an array of flip lengths $f_1,f_2, \ldots , f_{n+1}$ with each $f_i$ initialized to $i$.  For a formal proof of correctness, we invite the readers to see the simple inductive proof for the non-coloured case in~\cite{zaks}, and note the primary changes required to generalize to coloured permutations are in handling of the minimum allowable flip lengths (when $k=1$, the smallest allowable flip length is 2) corresponding to lines 5-6 and adding a factor of $k$ to line 8.  
%A complete C implementation of this iterative generation approach is provided in the appendix.
%
%-----------------------
\begin{algorithm}[h]
  \caption{Loop-free generation of the flip-sequence $\SEQQ^k_{n}$ }
  \label{alg:iter}
    \small
  \begin{algorithmic}[1]

    \Procedure{FlipSeq}{} 
    
    		\State $c_1,c_2, \ldots ,c_{n+1} \gets 0,0,\ldots , 0$
		\State $f_1,f_2, \ldots , f_{n+1} \gets 1,2, \ldots , n+1$

	\Repeat

		%\State \Call{Visit}{$\mathbf{p}$}

		  \If{$k = 1$}   $x \gets f_{2}$;  \ \   $f_{2} \gets 2$
		\Else \ \ $x \gets f_{1}$;  \ \   $f_{1} \gets 1$
		\EndIf
 		 \State $c_x \gets c_x + 1$ 
  		\If{ $c_x = kx{-}1$ } 
   			 \State $c_x \gets 0$
   			 \State $f_x \gets f_{x+1}$
   			 \State $f_{x+1} \gets x+1$
 		 \EndIf 
  		\State \Call{Output}{$x$}
		
		%\State \Call{Visit}{$\mathbf{p}$}
  
		%\State $\mathbf{p} \gets \flip{j}(\mathbf{p})$ 
	\Until $x > n$ 
    \EndProcedure
  \end{algorithmic}
\end{algorithm}
%-----------------------

\vspace{-0.2in}

\begin{theorem}
The algorithm  {\sc FlipSeq} is a loop-free algorithm to generate the flip-sequence  $\SEQQ^k_{n}$ one element at a time.
\end{theorem}

Since the average flip length in $\SEQQ^k_{n}$  is bounded by a constant, as determined in the previous subsection, Algorithm~\ref{alg:iter} can be modified to generate $\RecCMin{\pii}{k}$ by passing the initial permutation $\pii$ as a parameter, outputting $\pii$ at the start of the {\bf repeat} loop,  and updating 
$\pii \gets \flip{x}(\pii)$ at the end of the loop instead of outputting the flip length.

\begin{corollary}
The algorithm {\sc FlipSeq}  can be modified to generate successive permutations in the listing $\RecCMin{\pii}{k}$  in $O(1)$-amortized time.
\end{corollary}

%===============
%===============

%============================================
%==================================================
%============================================================
\section{Summary and Related Work}
\label{sec:summary}

We presented four different combinatorial algorithms for traversing a specific Hamilton cycle in the $k$-sided pancake network.
%:  a min-flip greedy algorithm requires exponential space to store the network, a recursive construction traverses the network in $O(1)$-amortized time using $O(n)$-space,  a successor-rule that allows the cycle to be traversed starting from any initial permutation (which the previous approaches do not allow for) in $O(1)$-amortized time per permutation, and a loop-free algorithm is given for the associated flip-sequence.  
The Hamilton cycle corresponds to a flip Gray code listing of coloured permutations.  Given such combinatorial listings, it is desirable to have associated ranking and unranking algorithms.  Based on the recursive description of the listing in \eqref{eq:bminflip}, such algorithms are relatively straightforward to derive and implement in $O(n^2)$-time.  Details are outlined in Appendix~\ref{app:rank}, along with a C implementation of our algorithms in Appendix~\ref{app:code}.

\bibliographystyle{abbrv}
\bibliography{refs}

\newpage
\appendix

\section{Ranking and unranking} \label{app:rank}
	
%===========================
\subsection{Efficient Ranking}

Let $\RANK{\pii}$ denote the rank of permutation $\pii =p_1p_2\cdots p_n \in \CPERMS{n}{k}$ in the listing $\RecCMin{1^02^0\cdots n^0}{k}$.  The rank of $\pii$ can be computed by considering the recursive structure of this listing given in \eqref{eq:bminflip}.  Each of the $nk$ recursive sub-lists contain $k^{n-1} \cdot (n-1)!$ permutations and there are $n(c_n+1) - p_n$ sub-lists that appear before the sub-list containing all permutations ending with $p_n$.  
%Thus, the first permutation ending with $p_n$ appears at rank $(n(c_n+1)-v_n) \cdot k^{n-1} \cdot (n-1)!$ + 1.   
Since each recursive sub-list follows the same structure, using appropriate relabelling we map $p_1p_2\cdots p_{n-1}$ to the permutation $q_1q_2\cdots q_n$ that appears at the same rank in the respective sub-list starting with $1^02^0\cdots (n{-}1)^0$.   In other words $\pii$ appears at rank $(n(c_n+1)-v_n) \cdot k^{n-1} \cdot (n-1)! + \RANK{q_1q_2\cdots q_{n-1}}$.  Formally:

\small
\begin{equation*}  \label{eq:rank} 
\RANK{p_1p_2\cdots p_n} =
\begin{cases}
c_i+1 & \text{ if $n=1$;} \\
%  (n-p_n) \ \  \cdot k^{n-1} \cdot (n-1)!\ + \ \RANK{r_1r_2\cdots r_{n-1}}  &  \text{ if $n>1$ and $p_n > 0$} \\
  (n(c_n{+}1){-}p_n) \cdot k^{n-1} (n-1)!  \ + \ \RANK{q_1q_2\cdots q_{n-1}}  &  \text{ otherwise}, \\
\end{cases}  
\end{equation*}  \normalsize
where $q_i = (v'_i,c'_i$) and $v'_i  = (v_i - v_n) \bmod n$ and $c'_i = (c_i-c_t) \bmod k$ if $v_i < v_n$ and $c'_i = (c_i-c_t-1) \bmod k$ otherwise.   
%This mapping of each $p_i$ to $q_i$ is perhaps best understood with an example.
%

\begin{exam} \small
Consider the listing $\RecCMin{ \C{1}{1}\C{2}{1}\C{3}{1}}{3}$ and the coloured permutation $\pii = \C{1}{1}\C{3}{2}\C{2}{3}$.    Notice from Example~\ref{exam:greedy} that  $\C{1}{1}\C{3}{2}\C{2}{3}$ is the 12th permutation in the sub-list of permutations ending with $\C{2}{3}$.    Similarly $\C{2}{2}\C{1}{2}\C{3}{1}$ is the 12th permutation in the first sublist ending with $\C{3}{1}$.  The first permutation in the listing that ends with $\C{2}{3}$ is $\C{3}{1}\C{1}{3}\C{2}{3}$.  By relabelling 
this permutation to correspond to $\C{1}{1}\C{2}{1}\C{3}{1}$ we consider the shifted sequence $r_1\cdots r_{27}$ to determine how to relabel each element $q_i$. 

\begin{center}
\begin{tabular}{c|| c | c | c | c | c | c | c | c  |  c  } 
$r_1r_2\cdots r_{27}$ & \C{1}{1} \C{2}{1} \C{3}{1} & \C{3}{2} \C{1}{1} \C{2}{1} &\C{2}{2} \C{3}{2} \C{1}{1} & \C{1}{2} \C{2}{2} \C{3}{2} &\C{3}{3} \C{1}{2} \C{2}{2} & \C{2}{3} \C{3}{3} \C{1}{2} & \C{1}{3} \C{2}{3} \C{3}{3} & \C{3}{1} \C{1}{3} \C{2}{3} & \C{2}{1} \C{3}{1} \C{1}{3} \\ \hline
shifted $r_1r_2\cdots r_{27}$  &\C{2}{2} \C{3}{2} \C{1}{1} & \C{1}{2} \C{2}{2} \C{3}{2} &\C{3}{3} \C{1}{2} \C{2}{2} & \C{2}{3} \C{3}{3} \C{1}{2} & \C{1}{3} \C{2}{3} \C{3}{3} & \C{3}{1} \C{1}{3} \C{2}{3} & \C{2}{1} \C{3}{1} \C{1}{3} & \C{1}{1} \C{2}{1} \C{3}{1} & \C{3}{2} \C{1}{1} \C{2}{1}  \\
\end{tabular}
\end{center}
 It is not difficult to see how this shifting produces the mapping to the $q_i$.  From this table:
 $p_1 = \C{1}{1}$  maps to $q_1 = \C{2}{2}$, and 
 $p_2 = \C{3}{2}$  maps to $q_2 = \C{1}{2}$. 
The recursive decomposition for computing $\RANK{ \C{1}{1}\C{3}{2}\C{2}{3} }$ in the ordering $\RecCMin{ \C{1}{1}\C{2}{1}\C{3}{1}}{3}$ is as follows:
\begin{eqnarray*} 
\RANK{ \C{1}{1}\C{3}{2}\C{2}{3} } & = & (3 \cdot 3 - 2) \cdot 3^2\cdot 2!  + \RANK{  \C{2}{2}\C{1}{2} } \\
  				& = & 126 +  (2 \cdot 2 - 1) \cdot 3^1\cdot 1! +  \RANK{  \C{1}{3} } \\
  				& = & 126 + 9 + 3 =   138 
\end{eqnarray*}
\end{exam}

\noindent
To compute $\RANK{p_1p_2\cdots p_n}$ requires $n$ recursive calls, where $O(n)$ work is required at each call.

\begin{theorem}
The algorithm $\RANK{\pii}$ returns the rank of the  permutation $\pii$ in the listing  $\RecCMin{1^02^0\cdots n^0}{k}$  in $O(n^2)$ time.
\end{theorem}

%===========================
\subsection{Efficient Unranking}

To find the permutation $\pii =p_1p_2\cdots p_n \in \CPERMS{n}{k}$ at position $rank$ in the listing $\RecCMin{1^02^0\cdots n^0}{k}$, we essentially inverse the operations applied in the ranking procedure:  We recursively determine the last symbol given the current rank,  then appropriately re-label a permutation found (recursively) at a specific rank in $\RecCMin{1^02^0\cdots (n{-}1)^0}{k}$.   Pseudocode for such an unranking algorithm is as follows. %provided in Algorithm~\ref{alg:minunrankC}. 

%-----------------------
\begin{algorithm}[h]
 % \caption{Computing the permutation at position $rank$  in the listing $\RecCMin{12\cdots t}{k}$ ,  $t \geq 1$  }
%  \label{alg:minunrankC}
  \small
  \begin{algorithmic}[1]
  
      \Function{UnRank}{$rank,t$}

		\If{$t=1$} \ \Return $1^{rank-1}$ \EndIf

		\State $x \gets \lfloor \frac{rank - 1}{k^{t-1} (t-1)!} \rfloor$
		\State  $p_t  \gets  (t-(x \bmod t))^{\lfloor x/t \rfloor}$
	
		\State$ q_1q_2\cdots q_{t-1}  \gets $ \Call{UnRank}{$rank - xk^{t-1}(t-1)!, \  t{-}1$}

		\For{$j \gets 1$ {\bf to} $t-1$}
				 \State $p_j \gets 1 + ((v_j + v_t - 1) \bmod t)$
				 \If{$v_j < v_t$} $c_j \gets (c_j + c_t) \bmod k$ 
				 \Else     \ \  $c_j \gets (c_j + c_t+1) \bmod k$ 
				  \EndIf
		\EndFor

		\State \Return $p_1p_2\cdots p_t$
	
    \EndFunction
  \end{algorithmic}
\end{algorithm} 
%-----------------------
%

\begin{exam} \small
Consider the permutation $\pii = p_1p_2p_3$ at $rank=138$ in the listing $\RecCMin{ \C{1}{1}\C{2}{1}\C{3}{1}}{3}$.  Stepping through the function {\sc UnRank}(138,3), the variable $x$ is assigned $\lfloor 137/(9\cdot 2!) \rfloor = 7$.  Thus $p_3$ will be the last element in the 8th sublist from the recursive description in \eqref{eq:bminflip} which is given by $v_3 = 3 - (7 \bmod 3) = 2$ and $c_3 = \lfloor 7/3 \rfloor = 2$.  Thus $p_3 = \C{2}{3}$.   Subtracting the $7\cdot (9\cdot 2!) =  126$ permutations found in the first 7 sub-lists from the $rank=138$, we are interested in the 12th permutation in the 8th sub-list ending with $\C{2}{3}$.  This permutation is computed by recursively determining the 12th permutation in $\RecCMin{12}{3}$, which is  $\C{2}{2}\C{1}{2}$ and applying an appropriate relabelling to obtain $p_1p_2 =  \C{1}{1}\C{3}{2}$.  Thus $\pii = \C{1}{1}\C{3}{2}\C{2}{3}$.
\end{exam}

\begin{theorem}
The algorithm {\sc UnRank}($rank, n$) returns the permutation at position $rank$ in the listing 
$\RecCMin{1^02^0\cdots n^0}{k}$  in  $O(n^2)$ time.
\end{theorem}

%===========================
%===========================
%===========================

\newpage
\section{Implementation of Algorithms in C} \label{app:code}

        \scriptsize
\begin{code}
#include <stdio.h>
#include <stdlib.h>
#define MAX_N 20

int n, k, a[MAX_N], color[MAX_N], f[MAX_N], c[MAX_N];
long int rank, total, powK[MAX_N], factorial[MAX_N];

//-------------------------------------------------
// OPERATIONS ON PERMUTATIONS
//-------------------------------------------------
void Flip(int t) {
    int i, b[MAX_N];

    for (i=1; i<=t; i++) b[i] = a[t-i+1];
    for (i=1; i<=t; i++) a[i] = b[i];

    for (i=1; i<=t; i++) b[i] = color[t-i+1];
    for (i=1; i<=t; i++) color[i] = (b[i]+1) % k;
}
//----------------------------------------
void Print() {
    int i;
    for (i=1; i<=n; i++)  printf("(%d,%d) ", a[i], color[i]);
    total++;
}
//----------------------------------------
// RANKING
//----------------------------------------
long int Rank(int t){
    int j;

    if (t == 1) return color[1] + 1;

    for (j=1; j<t; j++)  {
        if (a[j] <  a[t]) color[j] = (color[j] - color[t] + k) % k;
        if (a[j] >  a[t]) color[j] = (color[j] - color[t] - 1 + k) % k;
        a[j] = (a[j] - a[t] + t) % t;
    }
    return (( (color[t]+1) * t - a[t]) * factorial[t-1] * powK[t-1] + Rank(t-1));
}
//----------------------------------------
// UNRANKING
//----------------------------------------
void UnRank(long int rank, int t){
    int j,x;

    if (t == 1) {
        a[1] = 1;
        color[1] = rank - 1;
        return;
    }

    x = (rank-1) / (factorial[t-1] * powK[t-1]);
    rank = rank - x * factorial[t-1] * powK[t-1];

    // x ranges from 0 to kn-1
    a[t] = t - (x % t);
    color[t] = (int) x/t;

    UnRank(rank, t-1);

    for (j=1; j<t; j++)  {
        a[j] =  1 + (a[j] + a[t] - 1) % t;
        if (a[j] <  a[t]) color[j] = (color[j] + color[t]) % k;
        if (a[j] >  a[t]) color[j] = (color[j] + color[t] + 1) % k;
    }
}
//----------------------------------------
// SUCCESSOR
//----------------------------------------
void Successor() {
    int j, incr=0;

    for (j=1; j<n; j++) {
        if (a[j] < a[j+1]) incr++;
        if (incr == 2 || (incr == 1 && a[j+1] < a[1])) break;

        if (k > 1 && a[j] < a[j+1] && (color[j+1] - color[j] + k) % k != 1) break;
        if (k > 1 && a[j] > a[j+1] &&  color[j+1] != color[j]) break;
    }
    Flip(j);
}
//----------------------------------------
// GENERATION - BY FLIP SEQUENCE
//----------------------------------------
int Next() {
    int j;

    if (k == 1) { j = f[2];   f[2] = 2;  }
    else        { j = f[1];   f[1] = 1;  }

    c[j]++;
    if (c[j] == k*j-1) {
        c[j] = 0;
        f[j] = f[j+1];
        f[j+1] = j+1;
    }
    return(j);
}
//----------------------------------------
void Gen() {
    int j;

    // INITIAL PERM
    for (j=1; j<=n; j++) a[j] = j;
    for (j=1; j<=n; j++) color[j] = 0;
    
    for (j=1; j<=n+1; j++) c[j] = 0;
    for (j=1; j<=n+1; j++) f[j] = j;

    do {
        Print();  printf("\n");
        j= Next();
        Flip(j);

    } while (j <= n);

    printf("Total = %ld\n\n", total);
}
//================================================================================
int main() {
    int i,j,type;

    //------------------------------------
    // INPUT
    //------------------------------------
    printf("  1. Generation by flips \n");
    printf("  2. Ranking \n");
    printf("  3. UnRanking \n");
    printf("  4. Successor rule \n");
    printf("\n ENTER selection #: ");    scanf("%d", &type);

    if (type < 0 || type > 4)  { printf("\n INVALID ENTRY\n\n");  exit(0); }

    // long int constraints: MAX n seems to be 20 or 16 (k=2)
    printf(" ENTER order n: ");   scanf("%d", &n);
    printf(" ENTER colors k: ");  scanf("%d", &k);

    factorial[0] = 1;
    for (j=1; j< MAX_N; j++) factorial[j] = factorial[j-1] * j;
    powK[0] = 1;
    for (j=1; j< MAX_N; j++) powK[j] = powK[j-1] * k;

    if (type == 2 || type == 4) {
        printf(" ENTER coloured permutation of form:  p_1,c_1  p_2,c_2  ... ");
        for (i=1; i<=n; i++) scanf("%d,%d", &a[i], &color[i]);
    }
    if (type == 3) {
        rank = factorial[n] * powK[n];
        printf(" ENTER rank (between 1 and %ld): ", rank);
        scanf("%ld", &rank);
    }
    printf("\n");
    //------------------------------------

    if (type == 1) Gen();
    if (type == 2) {
        printf("The rank of permutation ");
        Print();
        rank = Rank(n);
        printf(" is:  %ld\n\n", rank);
    }
    if (type == 3) {
        UnRank(rank,n);
        printf("The permutation of rank %ld is: ", rank);
        Print(); printf("\n");
    }
    if (type == 4) {
        printf("The successor of permutation ");
        Print();
        Successor();
        printf(" is:  ");
        Print(); printf("\n");
    }
}
\end{code}

\end{document}